\RequirePackage{fix-cm}
\documentclass{svjour3}
\usepackage{blindtext}

\usepackage{graphicx}
\usepackage{amsmath,amsfonts}
\usepackage{amssymb,latexsym}
\usepackage[colorlinks]{hyperref}
\usepackage{graphicx}
\usepackage[all]{xy}
\usepackage{layout}
\usepackage[bottom]{footmisc}
\usepackage{mathrsfs}

\usepackage{tikz-cd}

\begin{document}
\renewcommand*{\max}{\mathrm{max}}

\newcommand{\C}{\mathbb{C}}
%
% \usepackage{mathptmx}      % use Times fonts if available on your TeX system
%
% insert here the call for the packages your document requires
%\usepackage{latexsym}
% etc.
%
% please place your own definitions here and don't use \def but
% \newcommand{}{}
%
% Insert the name of "your journal" with
%\journalname{Proceedings of the Royal  Society of Edinburgh}
%

	\title{Uniqueness of trace and C*-simplicity beyond regular representation}
	
	
	
	%\author{Massoud Amini \and Ali Shakibazadeh}
	\author{Massoud Amini}
	
	
	%\institute{M. Amini, Ali Shakibazadeh\at
	%             Faculty of Mathematical Sciences, Tarbiat Modares University\\ Tehran 14115-134, Iran
	%            \email{mamini@modares.ac.ir, a.shakibazadeh@modares.ac.ir}           %  \\
	%
	%}

	\institute{{\bf M. Amini}\at
		STEM Complex, 150 Louis-Pasteur Pvt,
		Ottawa, ON, Canada K1N 6N5\\
		\email{samin089@uottawa.ca}
}
	
	\date{Received: date / Accepted: date}

\maketitle
	
	\begin{abstract}
	A discrete group $\Gamma$ is C*-simple if the C*-algebra $C_\lambda^*(\Gamma)$ generated by the range of the left regular representation $\lambda$ on $\ell^2(\Gamma)$ is simple. In this case, $\Gamma$ acts faithfully on the Furstenberg boundary $\partial_F\Gamma$ and there is a unique trace on $C_\lambda^*(\Gamma)$. In this paper we study the unique trace property for the C*-algebra $C_\pi^*(\Gamma)$ generated by the range of an arbitrary unitary  representation $\pi: \Gamma\to B(H_\pi) $ and relate it to the faithfulness of the action of $\Gamma$ on the Furstenberg-Hamana boundary $\mathcal B_\pi$. Similar relation is obtained between simplicity of  $C_\pi^*(\Gamma)$ and (topological) freeness of the action of $\Gamma$ on $\mathcal B_\pi$. Along the way, we extend the Connes-Sullivan and Powers averaging properties for a unitary representation $\pi$ and relate them to simplicity and unique trace property of $C^*_\pi(\Gamma)$.

		
	\end{abstract}

\section{Introduction} \label{int}

The notion of boundary actions were introduced by Furstenberg \cite{f1} \cite{f2}, as a tool to study rigidity properties of
semisimple Lie groups. Recently, the Furstenberg idea has reemerged in the study of  rigidity properties of  reduced group C*-algebras generated by the regular representation of discrete groups. A discrete group $\Gamma$ acts canonically on its Furstenberg boundary $\partial_F\Gamma$ and the
dynamical properties of this action is known to be related to problems of simplicity, uniqueness of trace, and nuclear
embeddings \cite{bkko}, \cite{kk}, \cite{l}. The link between boundary actions and reduced
group C*-algebra was confirmed after the basic observation of Kalantar and Kennedy \cite{kk} that the Furstenberg and Hamana boundaries coincide, namely $C(\partial_F\Gamma) = I_\Gamma(\mathbb C)$,
where the right hand side is the $\Gamma$-injective envelope of complex numbers in the sense of
Hamana \cite{h2}. The properties of $\partial_F\Gamma$ are related to certain problems of interest in operator algebras. It is related to the
Ozawa's notion of ``tight nuclear embedding'' \cite{o2} of an
exact C*-algebra in a nuclear C*-algebra, which in turn embeds in the injective envelop of the original exact C*-algebra. Ozawa
proved such a tight embedding for the reduced C*-algebra of the free group $\mathbb F_n$, and Kalantar and Kennedy extended his result to an arbitrary discrete exact group $\Gamma$, by showing that the reduced C*-algebra $C_\lambda^*(\Gamma)$ of $\Gamma$ tightly embeds in the (nuclear) C*-algebra
generated by  $C_\lambda^*(\Gamma)$ and $C(\partial_F\Gamma)$ in $B(\ell^2(\Gamma))$.
The other problems are those of C*-simplicity and unique trace property for the reduced C*-algebra. Again it
is already known by a result of Powers \cite{po} that the free group $\mathbb F_2$ is
C*-simple and has the unique trace property. More general cases where handled using dynamical properties of the Furstenberg boundary: a
discrete group $\Gamma$ is C*-simple iff its action on the Furstenberg boundary $\partial_F\Gamma$
is topologically free \cite{kk} and has
the unique trace property iff this action is
faithful \cite{bkko}, and there are faithful, non free actions on the Furstenberg boundary \cite{l} (note that the first statement fails with ``topological freeness'' replaced by ``freeness'': indeed, given a non-C*-simple group $G$ with the unique trace property, for the canonical action of $G$ on its Furstenberg boundary, let $H$ be stabilizer of a point, which is known to be amenable. Then the C*-algebra generated by the
quasi-regular representation of $G/H$ is simple by a result of Kawabe, but
the Hamana-Furstenberg boundary of $G/H$ coincides with the Furstenberg boundary
of G, and in particular, the action of $G$ is not free.) This suggests that ``non triviality''
of the action on the Furstenberg boundary $\partial_F\Gamma$ (such as  being faithful or free) somehow measures the ``distance'' of $\Gamma$ from being amenable (c.f. \cite[Introduction]{bek}).

In an attempt to formulate the relation between ``boundaries'' with rigidity properties of the C*-algebra generated by the range of arbitrary representations, Bearden and Kalantar  associated a boundary $\mathcal B_\pi$ to  a unitary representation $\pi$  of $\Gamma$. They  related  properties of $\mathcal B_\pi$ (and action of $\Gamma$ on that) to rigidity properties of the C*-algebra $C_\pi^*(\Gamma)$ generated by the range of $\pi$. Bearden and Kalantar defined $\mathcal B_\pi$ as a
``relative'' Hamana $\Gamma$-injective envelop of the inclusion $\mathbb C{\rm id}_{H-\pi}\subseteq B(H_\pi)$,  naturally identified
with a $\Gamma$-invariant subspace of $B(H_\pi)$ and used the Choi-Effros product to make $\mathcal B_\pi$ into a C*-algebra. Back to the discussion of the previous paragraph, since amenability of $\Gamma$ characterizes by the existence of a translation invariant unital positive linear map: $\ell^\infty(\Gamma)\to \mathbb C$, given an arbitrary unitary representation $u$ of $\Gamma$, it is natural to seek for a ``minimal''
unital $\Gamma$-equivariant projection: $B(H_\pi)\to B(H_\pi)$, where now positivity is naturally replaced by complete positivity.
This is inline with the basic idea of Hamana in constructing the $\Gamma$-boundary \cite[Theorem 3.11]{h2} and
the idea of Bekka to define amenability of unitary representations \cite{b}. By such an analogy, Bearden and Kalantar defined $\mathcal B_\pi$ as the range of a minimal u.c.p. $\Gamma$-equivariant
projection.

In this paper, we show that the Furstenberg-Hamana boundary could be used to characterize simplicity and unique trace property of $C_\pi^*(\Gamma)$. The paper is organized as follows. In section \ref{pre} we recall the needed basic notions and concepts. Section \ref{bdry} reviews the construction of the Furstenberg-Hamana boundary $\mathcal B_\pi$ of a unitary representation $\pi$ of a discrete group $\Gamma$, used extensively in the rest of the paper. In section \ref{utp}, we study the uniqueness of trace on $C^*_\pi(\Gamma)$ and show that when $\pi$ is weakly contained in the left regular representation, this holds exactly when $\Gamma$ acts faithfully on $\mathcal B_\pi$. In this situation, if moreover $\mathcal B_\pi$ is C*-embeddable in the sense of \cite{kk}, we show in section \ref{sim} that $C^*_\pi(\Gamma)$ is simple only if  $\Gamma$ acts topologically free on $\mathcal B_\pi$, and the converse also holds when $\Gamma$ is exact. This is shown to be related to the Connes-Sullivan property for $\pi$. Section \ref{ex} motivates the paper by providing concrete examples of the situations where certain assumptions in our results hold for some classes of representations, including those coming from boundary actions or groupoid representations. This section also anticipates section \ref{ncbdry}, in which the classical $\Gamma$-boundaries are replaced by $\Gamma$-C*-algebras and topological freeness of the $\Gamma$-action on these noncommutative boundaries is explored.

\section{Preliminaries} \label{pre}
Let $H$ be a complex Hilbert space. A closed self-adjoint subspace $V$ of $B(H)$ is an {\it operator system} if it contains the identity. A linear map $\phi: V\to W$ between operator systems completely positive (c.p.) if all of its amplifications $\phi^{(n)}: \mathcal M_n(V)\to\mathcal M_n(W)$ are positive in the canonical operator space structure on the matrix algebras in the domain and range, and completely
isometric (c.i.) if these are all isometric. The unital completely positive (u.c.p.), unital completely
isometric (u.c.i),  and contractive completely positive (c.c.p.) maps are defined naturally. The c.c.p. maps are known to be completely contractive (c.c) as well. For a discrete group $\Gamma$, an operator system $V$ is a $\Gamma$-operator system if there is an action of $\Gamma$ on $V$, that is, a group homomorphism from $\Gamma$ to the group of all bijective u.c.i. maps on $V$.
A linear map  $\phi: V\to W$ between operator systems is a $\Gamma$-map if it is
u.c.p. and $\Gamma$-equivariant. A $\Gamma$-embedding is a c.i. $\Gamma$-map. When $W\subseteq V$, a $\Gamma$-projection is an idempotent $\Gamma$-map.

Let $\mathcal C$ be a subcategory of the category of operator systems and u.c.p. maps. An object
$V$ in $\mathcal C$ is injective  if for every c.i. morphism $\iota: U \to W$ and every morphism $\psi : U \to V$ there is a morphisms $\tilde\psi : W \to V$ with $\tilde\psi\circ\iota=\psi$. We are particularly interested in the subcategory of $\Gamma$-operator
systems and $\Gamma$-maps. Hamana has shown that if $V$ is injective operator system then  the  $\Gamma$-operator system $\ell^\infty(\Gamma, V)$ is $\Gamma$-injective \cite[Lemma 2.2]{h2}.

Let $V$ be a $\Gamma$-operator systems. A $\Gamma$-extension of $V$ is a pair $(W , \iota)$ consisting of
a $\Gamma$-operator system $W$ and a c.i. morphism $\iota : V \to W$.
A $\Gamma$-extension $(W , \iota)$ of $V$ is $\Gamma$-injective if $W$ is $\Gamma$-injective. It is $\Gamma$-essential if
for every morphism $\phi : W\to U$ such that $\phi\circ\iota$ is c.i. on $V$, $\phi$ is c.i. on $W$. It is $\Gamma$-rigid if for every morphism $\phi: W\to W$ such that
$\phi\circ\iota=\iota$ on $V$, $\phi$ is the identity map on $W$.

A $\Gamma$-extension of $V$ which is both $\Gamma$-injective and $\Gamma$-essential is called  a $\Gamma$-injective envelope of $V$. Every  $\Gamma$-injective envelope of $V$ is known to be $\Gamma$-rigid \cite[Lemma 2.4]{h}.

Every $\Gamma$-operator system $V$ has an injective $\Gamma$-injective
envelope, denoted by $I_\Gamma(V)$, which is unique up to c.i. $\Gamma$-isomorphism. A unital $\Gamma$-C*-algebra is $\Gamma$-injective as a C*-algebra if and only if it is $\Gamma$-injective as a $\Gamma$-operator system. Since $\ell^\infty(\Gamma, B(H))$ is $\Gamma$-injective,  if $V \subseteq B(H)$ is a $\Gamma$-operator system, then we have a $\Gamma$-embedding of $I_\Gamma(V)$  into $\ell^\infty(\Gamma, B(H))$ and a
$\Gamma$-projection: $\ell^\infty(\Gamma, B(H))\to I_\Gamma(V)$, which makes $I_\Gamma(V)$ into a C*-algebra via the Choi–Effros product.

Let $A$ be a C*-algebra and $V\subseteq A$ be an operator subsystem. If there is a surjective idempotent u.c.p. map $\psi : A \to V$, then the
Choi-Effros product $a \cdot b = \psi(ab)$, turns $V$ into a C*-algebra, which is unique up to isomorphism (does
not depend on the map $\psi$) \cite{ce}.

\section{boundary of representations}\label{bdry}
Let $\pi: \Gamma\to U(H_\pi)$ be a unitary representation of $\Gamma$ and consider the set
$$ \mathfrak M=\lbrace \phi:B(H_\pi)\rightarrow B(H_\pi):  \phi\ {\rm is\ a} \ \Gamma-{\rm map} \rbrace. $$

The set of $\mathfrak M $ with the partial order
$$ \phi\leq\psi \Longleftrightarrow \Vert \phi(x) \Vert \leq \Vert \psi(x) \Vert $$
contains a minimal element $\phi_0$, satisfying the following properties \cite{bek}:

$(i)$
$ \phi_{0}  $ is an idempotent.

$(ii)$ ($\pi$-essentiality)
Every $\Gamma$-map $ \psi: {\rm Im}(\phi_{0}) \rightarrow B(H_\pi) $   is isometric.

$(iii)$
${\rm Im}(\phi_{0}) \subseteq B(H_\pi) $ is minimal among subspace of $ B(H_\pi) $ containing $\mathbb C1$ that are images of $\Gamma$-projections.

$(iv)$ ($\pi$-rigidity)
The identity map is the unique $\Gamma$-map on $ {\rm Im}(\phi_{0})$.

$(v)$ ($\pi$-injectivity)
If $ X \subseteq Y $ are $\Gamma$-invariant subspaces of $ B(H_\pi) $ and $ \psi: B(H_\pi) \rightarrow B(H_\pi) $ is a $\Gamma$-map such that $ \psi(X) \subseteq {\rm Im}(\phi_{0}) $ then there is a $\Gamma$-map $ \tilde{\psi}: B(H_\pi) \rightarrow B(H_\pi) $ such that $ \hat{\psi}(Y) \subseteq {\rm Im}(\phi_{0}) $ and $ \tilde{\psi}=\psi $ on $X$.

In particular, the image of a minimal element of $\mathfrak M $ is unique up
to isomorphism. Now ${\rm Im}( \phi) $ is a $\Gamma$-C*-algebra under the Choi-Effros product, here called the
Furstenberg-Hamana boundary (or simply boundary) of the representation $\pi$, and is denoted by
$\mathcal{B}_{\pi}$.

The notion of amenable unitary representation is introduced by Bekka \cite[Definition 1.1]{b}. A unitary representation $\pi: \Gamma\to B(H)$ is amenable if there is a unital positive $\Gamma$-projection
$\phi : B(H_\pi) \to \mathbb C$, where $\Gamma$ acts on $B(H_\pi)$ by $Ad_\pi$. This is to say that the inclusion $\mathbb C1\subseteq B(H_\pi)$ is $\Gamma$-injective. In this case, the Furstenberg-Hamana boundary is trivial, that is,  $ \mathcal{B}_{\pi}\cong \mathbb C$.
Another aspect of ``triviality'' is the triviality of action. 	Every element of $\Gamma$ with finite conjugacy class acts trivially on $ \mathcal{B}_{\pi}$  \cite[Proposition 3.17]{bek}.

It is natural to expect $ \mathcal{B}_{\pi}$ to behave naturally under the restriction and induction of representations, as well as weak containment of representations. Recall that  a representation $\pi$ is {\it weakly contained} in another  representation $\sigma$, writing $\pi\preceq \sigma$, if the identity map on $C_c(\Gamma)$ extends canonically to a *-homomorphism: $C_\pi^*(\Gamma)\to  C_\sigma^*(\Gamma)$.
There is a
	$\Gamma$-map from $\mathcal{B}_{\sigma}$ to $\mathcal{B}_{\pi}$  if and only if there is a $\Gamma$-map from $B(H_\sigma)$ to $B(H_\pi)$.  In particular,
	if $\pi\preceq \sigma$,  there is a $\Gamma$-map: $ \mathcal{B}_{\sigma}\to\mathcal{B}_{\pi}$.

The last observation shows that the boundary $\mathcal{B}_{\pi}$  is determined by the weak equivalence class of $\pi$; in other words, if two representations are weakly equivalent, then  the corresponding boundaries are isomorphic as $\Gamma$-C*-algebras. The particular case of $\pi\preceq \lambda$ is of special interest, where $\lambda: \Gamma\to U(\ell^2(\Gamma))$ is the left regular representation. In this case, the Furstenberg-Hamana boundary is the same as the Furstenberg boundary, namely, $\mathcal{B}_{\lambda}\cong C(\partial_F\Gamma)$ \cite[Example 3.7]{bek}.

An action of $\Gamma$ on a compact space $X$ is {\it topologically amenable}  if there is a net of continuous maps
$m_i : X \to$ Prob$(\Gamma)$ such that $\Vert gm_i(x) - m_i(gx)\Vert_1\to 0$, as $i\to\infty$, for every $g\in \Gamma$ \cite[Definition 2.1]{d}. Note that the group $\Gamma$ is exact if and only if it admits a topologically amenable action
on a compact space \cite{o}. Also for compact $\Gamma$-spaces $X$ and $Y$ if there is a continuous $\Gamma$-equivariant map:$ Y \to X$ and
$X$ is topologically amenable
then so is $Y$.

In general, there is a
	$\Gamma$-inclusion $\mathcal{B}_{\pi}\hookrightarrow I(C^*_{\pi}(\Gamma))$. The boundary $\mathcal{B}_{\pi}$ is C*-embeddable if there is a *-homomorphic copy of
$\mathcal{B}_{\pi}$ in $B(H_\pi)$. In this case,  $\tilde{\mathcal{B}}_{\pi}$ denotes the C*-algebra generated by $\mathcal{B}_{\pi}\cup \pi(\Gamma)$
in $B(H_\pi)$.

An action of $\Gamma$ on a unital C*-algebra $B$ is (strongly)
topologically amenable if  the action of $\Gamma$ on the Gelfand spectrum of the
center $Z(B)$ of $B$ is topologically amenable (c.f., \cite{d}). A related notion is defined in \cite[Definition 3.4]{bew2}: an action of $\Gamma$ on a  C*-algebra $B$ is strongly
 amenable if there is a net  of norm-continuous, compactly supported, positive type functions $\theta_i : \Gamma \to ZM(A)$ such that $\|\theta_i(e)\|\leq 1$
for all $i$ (where $e$ is the identity of $\Gamma$), and $\theta_i(g)\to 1$,  strictly and uniformly for $g$ in compact
subsets of $\Gamma$. The two notions are equivalent when $B$ is commutative (and unital) \cite[Theorem 5.13]{bew2}. Note that in \cite{bew2} the authors also define the notion of amenable action \cite[Definition 3.4]{bew2}, which is in general weaker than strong amenability.

	If $\mathcal{B}_{\pi}$ is  C*-embeddable and the action of $\Gamma$ on $\mathcal{B}_{\pi}$ is topologically amenable, then $\Gamma$ is an exact group and $\pi\preceq\lambda$. Conversely, if
	$\pi\preceq\lambda$,  then the action of $\Gamma$ on $\mathcal{B}_{\pi}$ is topologically amenable.

We also have the following characterization of strong amenability of the $\Gamma$-action on $\mathcal{B}_{\pi}$ which extends \cite[Theorem
1.1]{kk}: the action of $\Gamma$ on $\mathcal{B}_{\pi}$ is strongly  amenable iff $\Gamma$ is exact.

Following \cite{bek}, for a normal subgroup $\Lambda\unlhd \Gamma$, we say that $\Lambda$ is $\pi$-amenable if there is a $\Gamma$-map $\phi: B(H_\pi)\to \pi(\Lambda)^{'}$, where the commutant is taken inside $B(H_\pi)$ and is known to be $\Gamma$-invariant. The whole group $\Gamma$ is $\pi$-amenable iff $\pi$ is an amenable representation iff all normal subgroups of $\Gamma$ are $\pi$-amenable \cite[Proposition 4.2]{bek}. Also a normal
subgroup $\Lambda\unlhd \Gamma$ is $\lambda$-amenable iff it is amenable \cite[Proposition 4.3]{bek}.

A normal subgroup $\Lambda\unlhd \Gamma$ acts trivially on  $\mathcal{B}_{\pi}$ iff it is $\pi$-amenable. In particular, the kernel of the action  of $\Gamma$ on $\mathcal{B}_{\pi}$ is the unique maximal $\pi$-amenable
normal subgroup of $\Gamma$, which contains all $\pi$-amenable normal subgroups of $\Gamma$. This is called the $\pi$-amenable radical of
$\Gamma$, and is denoted by Rad$_{\pi}(\Gamma)$ \cite[Definition 4.6]{bek}.

Let $V$ be a $\Gamma$-operator system. Then the $\Gamma$-injective envelope $I_\Gamma(V)$ is characterized as a unique maximal
$\Gamma$-essential (or minimal $\Gamma$-injective) $\Gamma$-extension of $V$.
If $I(V)$ is the injective envelope of $V$ as an
operator system, since Aut$V\subseteq$ Aut$I(V)$, one may regard $I(V)$ as a $\Gamma$-extension of $V$. Indeed,  $I(V)$
is a $\Gamma$-essential $\Gamma$-extension of $V$ and $V\subseteq I(V)\subseteq I_\Gamma(V)$ as $\Gamma$-spaces. Moreover,  $I(V)$ is unique among the $\Gamma$-subspaces  of $I_\Gamma(V)$ which are the injective envelope of $V$ \cite[Remark 2.6]{h2}.

Recall that a normal subgroup $\Lambda\unlhd \Gamma$ is called coamenable if there is a $\Gamma$-state on  $\ell^\infty(\Gamma/\Lambda)$.  Every coamenable
	normal subgroup  $\Lambda\unlhd \Gamma$ is $\pi$-coamenable. The converse is also true when $\pi\preceq\lambda$. In particular, a
	normal subgroup  $\Lambda\unlhd \Gamma$ is coamenable iff it is $\lambda$-coamenable.
	
If $\pi$  restricts to an irreducible representation on $\Lambda$, then $\Lambda$ is $\pi$-coamenable, since in this case, $\pi(\Lambda)^{'}$ is trivial. If $\pi$ is amenable, then each
  	normal subgroup  $\Lambda\unlhd \Gamma$ is both $\pi$-amenable and $\pi$-coamenable. Conversely, if there is a normal subgroup  $\Lambda\unlhd \Gamma$ which is both $\pi$-amenable and $\pi$-coamenable, $\pi$ is amenable.
In particular, if $\pi$ is a non amenable
unitary representation of $\Gamma$ whose restriction to a normal subgroup $\Lambda$ is reducible, then $\Lambda$ is not $\pi$-amenable \cite[Corollary 4.13]{bek}.

The Furstenberg boundary has the following extension property in topological dynamics: if $\Lambda$ is a normal subgroup of $\Gamma$, then any action of $\Lambda$ on $\partial_F\Gamma$ extends to an action of $\Lambda$ on $\partial_F\Gamma$. There are analogous extension properties of the Furstenberg-Hamana boundary $\mathcal{B}_{\pi}$.
An automorphism $\tau\in {\rm Aut}(\Gamma)$ is
called a $\pi$-automorphism if it extends to a C*-automorphism of $C^*_{\pi} (\Gamma)$. Following \cite{bek}, we
denote the set of all such $\Gamma$-automorphisms by ${\rm Aut}_{\pi}(\Gamma)$. The automorphism on $C^*_{\pi} (\Gamma)$ induced by a $\Gamma$-automorphism $\tau$ leaves the kernel of the
canonical *-epimorphism: $C^*_{\rm max}(\Gamma)\to C^*_{\pi} (\Gamma)$ invariant. In particular, ${\rm Aut}_{\lambda}(\Gamma)={\rm Aut}(\Gamma)$.

Any inner automorphism is a $\pi$-automorphism, thus, $\Gamma/Z(\Gamma)$ is identified with a normal subgroup of
${\rm Aut}_{\pi}(\Gamma)$. Since $Z(\Gamma)$ acts trivially on  $\mathcal{B}_{\pi}$, it follows that the action of $\Gamma$ on $\mathcal{B}_{\pi}$ factors through $\Gamma/Z(\Gamma)$. Also, 
	The action of $\Gamma$ on $\mathcal{B}_{\pi}$ extends
	to an action of ${\rm Aut}_{\pi}(\Gamma)$
	on $\mathcal{B}_{\pi}$.

Recall that Rad$_{\pi}(\Gamma)$ is kernel of the action of $\Gamma$ on $\mathcal{B}_{\pi}$. Each $\tau\in{\rm Aut}_{\pi}(\Gamma)$ keeps Rad$_{\pi}(\Gamma)$ invariant. This induces a group homomorphism
$$\kappa: {\rm Aut}_{\pi}(\Gamma)\to {\rm Aut}_{\pi}\big(\Gamma/{\rm Rad}_{\pi}(\Gamma)\big),$$
whose kernel is nothing but the kernel of the above unique extended action. In particular, if the action of $\Gamma$ on $\mathcal{B}_{\pi}$ is faithful, then so is its unique extension to an action of ${\rm Aut}_{\pi}(\Gamma)$
on $\mathcal{B}_{\pi}$.

Let	 $\Lambda$ be a normal subgroup of $\Gamma$, and let $\pi_0$ be the restriction of $\pi$ on $\Lambda$. Then the action of $\Lambda$
on $\mathcal{B}_{\pi_0}$ extends to an action of $\Gamma$
on $\mathcal{B}_{\pi_0}$. When the $\Lambda$-action is faithful, the kernel of the $\Gamma$-action is the inverse image of the centralizer of $\pi(\Lambda)$ in $\pi(\Gamma)$.

\section{Uniqueness of Trace}\label{utp}

The groups with the unique trace property are studied in  \cite{bkko}.
 Every trace on the reduced C*-algebra $C_\lambda^*(\Gamma)$ of a discrete group $\Gamma$  is known to be supported on the amenable radical
Rad$(\Gamma)$ \cite[Theorem 4.1]{bkko}, which is the kernel of the action of $\Gamma$ on the Furstenberg boundary $\partial_F\Gamma$.

A trace  on $C^*_\lambda(\Gamma)$ is nothing but a $\Gamma$-map $\tau: C^*_\lambda(\Gamma)\to\mathbb C$, where
$\mathbb C$ carries the trivial action of $\Gamma$. By $\Gamma$-injectivity, $\tau$ extends to a $\Gamma$-map $\tilde\tau: B\big(\ell^2(\Gamma)\big)\to I_\Gamma(\mathbb C)=C(\partial_F\Gamma)$. Since $\mathcal{B}_{\lambda}=C(\partial_F\Gamma)$, we may regard $\tilde\tau$ as a $\Gamma$-map $\tilde\tau: B\big(\ell^2(\Gamma)\big)\to \mathcal{B}_{\lambda}$. Now a trace $\tau$ on $C^*_\lambda(\Gamma)$ is  supported on the normal subgroup Rad$_{\pi}(\Gamma)\unlhd \Gamma$ iff the extension
$\tilde\tau: B\big(\ell^2(\Gamma)\big)\to \mathcal{B}_{\lambda}$ vanishes outside the  kernel of the action of $\Gamma$ on $\mathcal{B}_{\lambda}$.

A trace on $C^*_\pi(\Gamma)$ is called {\it extendable} in \cite{bek}, if it extends to a $\Gamma$-map:  $B(H_\pi)\to  \mathcal{B}_{\pi,u}$. In particular, if $\pi\preceq\lambda$ then any trace on $C^*_\pi(\Gamma)$ is extendable.

\begin{lemma} \label{rad2}
	Let	
	$\psi: B(H_\pi) \to \mathcal{B}_{\pi}$ be a $\Gamma$-map and $g\notin${\rm Rad}$_{\pi}(\Gamma)$, then $\psi(\pi_g)$ is in the kernel of  some state on $\mathcal{B}_{\pi}$.
\end{lemma}
\begin{proof}
	Let $g\notin$Rad${}_{\pi}(\Gamma)$, and choose $x\in \mathcal{B}_{\pi}$ with $x\neq \pi_gx\pi_g^*$.
	We may assume that $0\leq x\leq 1$. Choose a state $\omega$ on $\mathcal{B}_{\pi}$ with $\omega(x)=1$ and $\omega(\pi_gx\pi_g^*)=0$.
	Since $\psi$ acts as identity on $\mathcal{B}_{\pi}$, $\omega\circ \psi(x)=1$ and $ \omega\circ\psi(\pi_gx\pi_g^*)=0$. For the state $\rho=\omega\circ \psi$, we  have $ \rho(\pi_gx^2\pi_g^*)= \rho(1-x^2)=0$,  and so by Cauchy-Schwartz inequality,  $ \rho\big((1-x)\pi_g\big)= \rho\big(x\pi_g\big)=0$, thus
	$$ \rho(\pi_g)=\rho\big((1-x)\pi_g\big)+ \rho\big(x\pi_g\big)=0,$$
	that is, $\psi(\pi_g)$ is in the kernel of $\omega$.
	\qed	
\end{proof}

A trace  on  $C^*_\pi(\Gamma)$ is nothing but a $\Gamma$-map $\tau:  C^*_\pi(\Gamma)\to\mathbb C$, where
$\mathbb C$ carries the trivial action of $\Gamma$ by Ad$_\pi$. We say that $\tau$ is supported on a normal subgroup $\Lambda\unlhd \Gamma$ if $\tau(\pi_g)=0$, for $g\notin \Lambda$.

\begin{proposition} \label{rad3}
	All extendable traces on $C^*_\pi(\Gamma)$ are supported on {\rm Rad}$_{\pi}(\Gamma)$.
\end{proposition}
\begin{proof}
Let $\tau$ be an extendable  trace on  $C^*_\pi(\Gamma)$ and consider the $\Gamma$-extension $\tilde\tau: B(H_\pi) \to \mathcal{B}_{\pi}$. Then $\tilde\tau(\pi_g)=\tau(\pi_g)1$, for any $g\in \Gamma$. If $g\notin$Rad$_{\pi}(\Gamma)$, then by Lemma \ref{rad2}, $\tilde\tau(\pi_g)$ is in the kernel some state $\omega$ on $\mathcal{B}_{\pi}$. Hence, $\tau(\pi_g) = \omega(\tau(\pi_g)1)=\omega(\tilde\tau(\pi_g))=0$.\qed
\end{proof}

\begin{corollary} \label{rad4}
If $\Gamma$ acts faithfully on $\mathcal{B}_{\pi}$, then  $C^*_\pi(\Gamma)$ has either has no extendable trace or has a unique one.
\end{corollary}
\begin{proof}
	If we have Rad$_{\pi}(\Gamma)=\{e\}$, then for an extendable trace $\tau$, $\tau(\pi_g)=0$, for $g\neq e$. Since an extendable trace on  $C^*_\pi(\Gamma)$ extend to a functional on $M( C^*_\pi(\Gamma))$, this means that any two traces  agree on the copy of $C^*_\pi(\Gamma)$ inside  $M(C^*_\pi(\Gamma))$.  \qed
\end{proof}

\begin{lemma}\label{can}
	The following are equivalent:
	
	$(i)$ there is a canonical trace $\tau_0$ on  $C^*_\pi(\Gamma)$ satisfying $\tau_0(\pi_t)=\delta_{t e}$, where the right hand side is the Kronecker delta,
	
	$(ii)$ for each normal subgroup $\Lambda\unlhd \Gamma$, there is  a conditional expectation  $\mathbb E_\Lambda: C^*_\pi(\Gamma)\to C^*_\pi(\Lambda)$, which is a $\Gamma$-map satisfying $\mathbb E_\Lambda(\pi_t)=1_\Lambda(t)\pi_t$, where $1_\Lambda$ is the characteristic function of $\Lambda$ considered as a function on $\Gamma$.
	
	\noindent Moreover, if $\lambda\preceq\pi$,  $\pi$ satisfies these  conditions. In this case, we also have:
	
	$(iii)$ for each amenable subgroup $\Lambda\leq \Gamma$, there is  a state on $C^*_\pi(\Gamma)$ extending the characteristic function $1_\Lambda$.
	
	\end{lemma}
 \begin{proof}
  $(i)\Rightarrow (ii)$. Extend the trace on  $C^*_\pi(\Gamma)$ to a normal trace $\tau$ on $M:=\pi(\Gamma)^{''}$. Then by standard results of von Neumann algebras (c.f, \cite[Proposition 11.21]{ps}), for the von Neumann subalgebra $N:=\pi(\Lambda)^{''}$, there is a conditional expectation: $M\to N$. Here we need the specific properties of this expectation, so we recall the proof from \cite{pi}. By definition $L^2(N,\tau|_N)$ can be naturally identified to a subspace of
$ L^2(M,\tau)$, namely the closure $\bar N$ of $N$ in $L^2(M,\tau)$. Let $P$ be the orthogonal
 projection from $L^2(M,\tau)$ to $\bar N \subseteq  L^2(N,\tau|_N)$. For $x \in L^2(M,\tau)$, $P(x)$ is the unique $\tilde x \in L^2(N,\tau|_N)$ with $\langle x,y\rangle = \langle \tilde x,y\rangle$ for each $y\in N$. By \cite[Remark 11.15]{pi}, $P$ is an $N$-module map, which is clearly acts as identity on $N$, in particular, it is a norm-one projection, and for $ a,b\in N$, the form $x\mapsto  \langle a,P(x)b\rangle = \tau(a^*xb)$ is normal, and so is $P$. The complete positivity of $P$ is now automatic by  \cite[Theorem 1.45]{pi}. Let  $\mathbb E_\Lambda$ be the restriction of $P$ to $C^*_\pi(\Gamma)$. Since the image of norm-closed subspaces under conditional expectations are norm-closed and $P$ acts as identity on $C^*_\pi(\Lambda)$, $\mathbb E_\Lambda$ maps $ C^*_\pi(\Gamma)$ onto a norm closed subspace containing $C^*_\pi(\Lambda)$. 
 
 Next, given $t\in\Lambda$ and $s\in\Gamma\backslash\Lambda$,
 $$\langle\pi_t,\pi_s\rangle=\tau(\pi^*_t\pi_s)=\tau_0(\pi_{t^{-1}s})=0,$$
 since $t^{-1}s\neq e$. Therefore, $\pi_s$ is perpendicular to the linear span of $\pi(\Lambda)$ in $L^2(M,\tau)$, which is norm dense in $L^2(N,\tau|_N)$. In particular, $P\pi_s=0$, for $s\in\Gamma\backslash\Lambda$. On the other hand,  
 $$\mathbb E_\Lambda(\pi_t)=P\pi_tP\ \ (t\in\Gamma),$$
 where $P$ is viewed as an orthogonal
 projection from $L^2(M,\tau)$ onto $L^2(N,\tau|_N)$. Thus 
  $$\mathbb E_\Lambda(\sum_{t\in\Gamma}a_t\pi_t)=\sum_{t\in\Lambda}a_t\pi_t,$$
  where the scalars $a_t$ are zero except for finitely many indices. 
 This implies that $\mathbb E_\Lambda$ maps $ C^*_\pi(\Gamma)$ onto  $C^*_\pi(\Lambda)$. 
 
 $(ii)\Rightarrow (i)$.
 Applying the assumption to the trivial subgroup, we get the canonical trace satisfying  
 $$\tau_0(\sum_{t\in\Gamma} a_t\pi_t):=a_e,$$
 where the scalars $a_t$ are zero except for finitely many indices.
 
 Next if  $\lambda\preceq\pi$ then there is a $\Gamma$-map $\phi: C^*_\pi(\Gamma)\to C^*_\lambda(\Gamma)$. This map composed with the canonical trace $x\mapsto \langle x\delta_e,\delta_e\rangle$ on $C^*_\lambda(\Gamma)$ is a trace $\tau_0$ on $C^*_\pi(\Gamma)$ which  satisfies condition $(i)$ above:  
 $$\tau_0(\pi_t)=\langle\phi(\pi_t)\delta_e,\delta_e\rangle=\langle\lambda_t\delta_e,\delta_e\rangle=\langle\delta_t,\delta_e\rangle=0,$$
 when $t\neq e$.
 
Finally, for an amenable subgroup $\Lambda\leq \Gamma$, there is a state $\rho$ on $C^*_\lambda(\Gamma)$ extending the characteristic function $1_\Lambda$ (see the proof of \cite[Proposition 3.2]{ke}). Since the above $\Gamma$-map $\phi: C^*_\pi(\Gamma)\to C^*_\lambda(\Gamma)$ satisfies $\phi(\pi_s)=\lambda_s$, for $s\in\Gamma$, $\rho\circ\phi$ is the desired state on $C^*_\pi(\Gamma)$.  \qed
\end{proof}

\vspace{.3cm}
In particular, the above equivalent conditions hold for the case of left regular representation.

\begin{theorem} \label{main}
	If   $\pi\preceq\lambda$ and there is a trace on  $C^*_\pi(\Gamma)$, the following are equivalent:
	
	$(i)$\ $\Gamma$ acts faithfully on $\mathcal{B}_{\pi}$, 
	
	$(ii)$\ $C^*_\pi(\Gamma)$ has unique trace property.
\end{theorem}
\begin{proof}
	 $(i)\Rightarrow (ii)$. Since $\pi\preceq\lambda$, all traces are extendable and this implication follow from Corollary \ref{rad4}.
	 
	 $(ii)\Rightarrow (i)$.  Let $\Lambda\unlhd \Gamma$ be a  $\pi$-amenable normal subgroup and $\phi: B(H_\pi)\to \pi(\Lambda)^{'}$ be the corresponding $\Gamma$-map. Since $\Lambda$ acts trivially on $\pi(\Lambda)^{'}$, the composition of $\phi$ with any state $\tau$ on $\pi(\Lambda)^{'}$ gives a $\Lambda$-map $\tau\circ\phi: B(H_\pi)\to \mathbb C$, which means that the trivial representation is weakly contained in $\pi|_{\Lambda}$, in particular, the unit character on $\Lambda$ that sends each $\pi_t$ to 1, for $t\in\Lambda$ extends
	 to a non-canonical trace $\tau_1$ on $C^*_\pi(\Lambda)$, unless $\Lambda$ is the trivial subgroup.  This composed with the conditional expectation $\mathbb E_\Lambda$ of Lemma \ref{can} gives a trace $\tau_1\circ \mathbb E_\Lambda$ on $C^*_\pi(\Gamma)$, which is different with the canonical trace $\tau_0$ of Lemma \ref{can}, which is a contradiction, forcing $\Lambda$ to be trivial. Therefore, Rad$_{\pi}(\Gamma)$ is trivial.  \qed
\end{proof}

Note that in the particular case where $\mathcal{B}_{\pi}\cong C(\partial_F\Gamma)$, canonically as $\Gamma$-spaces,  the assumption Rad$_{\pi,u}(\Gamma)=\{e\}$ follows from Rad$(\Gamma)=\{e\}$, since the latter means that $\Gamma$ acts faithfully on $\partial_F\Gamma$.

In order to study traces on $C^*_\pi(\Gamma)$, following \cite{ke}, we study $\Gamma$-boundaries inside the state space $\mathcal{S}(C^*_\pi(\Gamma))$. Recall that a $\Gamma$-space $X$ is a $\Gamma$-{\it boundary} if for each probability measure $\nu$ on $X$ and point $x\in X$, the Dirac measure $\delta_x$ is in the weak$^*$-closure of $\Gamma\nu$. A convex $\Gamma$-space $K$ is {\it affine} if the restriction of $\Gamma$-action to $K$ is implemented by convex maps, and is {\it minimal} if it is minimal among all affine $\Gamma$-spaces. By Zorn lemma, inside each affine $\Gamma$-space one has a unique minimal one. It is also a standard fact that for a minimal affine $\Gamma$-space $K$, the closure of the set ext$(K)$ is  a $\Gamma$-boundary \cite[III.2.3]{gl2}. 

Given a trace $\tau\in \mathcal{S}(C^*_\pi(\Gamma))$, since $\Gamma$ acts on $C^*_\pi(\Gamma)$ by $s\cdot a:=\pi_sa\pi_s^*$, the singleton $\{\tau\}$ is a $\Gamma$-boundary in $\mathcal{S}(C^*_\pi(\Gamma))$. Unlike the case of regular representation, $C^*_\pi(\Gamma)$ may fail to have a canonical trace (or indeed any trace at all), however, when $\lambda\preceq\pi$, by Lemma \ref{can}, we have a canonical trace $\tau_0$ on $C^*_\pi(\Gamma)$. In this case, we say that a $\Gamma$-boundary $X$ is {\it trivial} if $X=\{\tau_0\}.$ In particular, if the trace on $C^*_\pi(\Gamma)$ is not unique, then for any other trace $\tau$, the singleton $\{\tau\}$ is a non-trivial $\Gamma$-boundary, for the action via ad$_\pi$. 

we say that a $\Gamma$-boundary $X\subseteq \mathcal S(C^*_\pi(\Gamma))$ is a $\pi$-boundary if there is a weak$^*$-continuous  $\Gamma$-map: $\mathcal P(C^*_\pi(\Gamma))\to X$. Note that $\Gamma$-boundaries are automatically $\lambda$-boundary.  

Though in the proof of the next result we use a result proved in Section \ref{sim} (whose proof is independent of the results of this section), we prefer to keep it here. Recall that a boundary map is a u.c.p. $\Gamma$-map into $\mathcal B_\pi$.  

\begin{proposition} \label{1-1 cor}
	If $\pi\preceq\lambda$, then there is a one-one correspondence between boundary maps: $C^*_\pi(\Gamma)\to \mathcal{B}_{\pi}$ and $\pi$-boundaries inside $\mathcal S(C^*_\pi(\Gamma))$.
\end{proposition}
\begin{proof}
	Given a u.c.p. $\Gamma$-map $\phi: C^*_\pi(\Gamma)\to \mathcal{B}_{\pi}$, for the induced $\Gamma$-equivariant map $\phi: \mathcal S(\mathcal{B}_{\pi})\to \mathcal S(C^*_\pi(\Gamma))$, $K:=\phi^*(\mathcal P(\mathcal{B}_{\pi}))$ is a minimal affine $\Gamma$-space by \cite[III.2.4]{gl2}, and so $X:=$ext$(K)$ is a $\Gamma$-boundary, by the first statement of Proposition \ref{mp} and \cite[III.2.3]{gl2}. Since $\pi\preceq\lambda$, there is a $\Gamma$-map: $C^*_\lambda(\Gamma)\hookrightarrow C^*_\pi(\Gamma)$ whose adjoint map restricts to a  $\Gamma$-map: $\mathcal P(C^*_\pi(\Gamma))\to\mathcal P(C^*_\lambda(\Gamma))=\partial_F\Gamma$. By universality, there is a surjective $\Gamma$-map: $\partial_F\Gamma\twoheadrightarrow X$. Combining these maps, we get a $\Gamma$-map: $\mathcal P(C^*_\pi(\Gamma))\to X$. 
	
	Conversely, if $X\subseteq \mathcal S(\mathcal{B}_{\pi})$ is a $\pi$-boundary with a weak$^*$-continuous  $\Gamma$-map $\varphi: \mathcal P(C^*_\pi(\Gamma))\to X$, then we define a $\Gamma$-map $\phi: C^*_\pi(\Gamma)\to \mathcal{B}_{\pi}$ via, $$\langle\phi(a),\omega\rangle:=\langle a,\varphi(\omega)\rangle,\ \ (a\in C^*_\pi(\Gamma), \omega\in \mathcal P(\mathcal{B}_{\pi})),$$
	as required. Note that $\phi(a)$ extends affinely and continuously to a map on the state space and then to a weak$^*$-continuous map on the dual space of $\mathcal{B}_{\pi}$,  thereby defining an element of $\mathcal{B}_{\pi}$.    \qed	
\end{proof}	

In the next corollary, we do not assume that $C^*_\pi(\Gamma)$ has a trace (though this is the case in many concrete examples).
 
\begin{corollary}\label{conv}
	Assume that $\pi\preceq\lambda$. If there is a unique boundary map: $C^*_\pi(\Gamma)\to \mathcal{B}_{\pi}$, then for each $\omega\in\mathcal S(C^*_\pi(\Gamma))$, the closed convex hull of $\Gamma \omega$ contains all the traces on $C^*_\pi(\Gamma)$, namely, $\mathcal T(C^*_\pi(\Gamma))\subseteq \overline{\rm conv}(\Gamma \omega)$.  	
\end{corollary}
\begin{proof}
Given $\tau\in \mathcal T(C^*_\pi(\Gamma))$, then $\{\tau\}$ is a $\pi$-boundary and this is the only $\pi$-boundary by Proposition \ref{1-1 cor}.  Given $\omega\in\mathcal S(C^*_\pi(\Gamma))$, $\overline{\rm conv}(\Gamma \omega)$ contains a minimal affine $\Gamma$-space $K$ and $\overline{{\rm ext}(K)}$ is a $\Gamma$-boundary. Now, as in the proof of Proposition \ref{1-1 cor}, the assumption $\pi\preceq\lambda$ implies that $\overline{{\rm ext}(K)}$ is also a $\pi$-boundary. In particular, 
$$\{\tau\}=\overline{{\rm ext}(K)}\subseteq K\subseteq \overline{\rm conv}(\Gamma \omega),$$
i.e., $\tau\in \overline{\rm conv}(\Gamma \omega)$, as claimed.\qed
\end{proof}

The above result motivates the extension of the celebrated {\it Powers averaging property}  for the left regular representation  \cite{po} to any representation $\pi$. 

\begin{definition}
We say that a representation $\pi\in$Rep$(\Gamma)$ has Powers averaging property if for each $a\in C^*_\pi(\Gamma)$ and $\varepsilon>0$, there is $n\geq 1$ and there are elements $s_1,\cdots, s_n\in\Gamma$ such that,
$$\Big\|\frac{1}{n}\sum_{1}^{n} \pi_{s_i}a\pi_{s_i}^*-\tau(a)1\Big\|<\varepsilon,$$
for any  trace $\tau\in \mathcal T(C^*_\pi(\Gamma))$.   	
\end{definition}

Again we do not assume that $C^*_\pi(\Gamma)$ has a trace (if not, the above property is assumed to hold as a vacuous truth). When $\lambda$ has Powers averaging property, we say that $\Gamma$ has Powers averaging property \cite{po}. 

\begin{theorem}\label{pap}
If $\pi$ has Powers averaging property then $C^*_\pi(\Gamma)$ has unique trace property. If moreover $C^*_\pi(\Gamma)$ has at least one faithful trace, then $C^*_\pi(\Gamma)$ is also simple (i.e., $\Gamma$ is $C^*_\pi$-simple). Conversely, if $\pi\preceq\lambda$, $C^*_\pi(\Gamma)$ has at least one  trace, and there is a unique boundary map: $C^*_\pi(\Gamma)\to \mathcal{B}_{\pi}$,  then $\pi$ has Powers averaging property.
\end{theorem} 
\begin{proof}
If $\pi$ has Powers averaging property and there are two traces $\tau_1, \tau_2\in\mathcal T(C^*_\pi(\Gamma))$, then for $a\in C^*_\pi(\Gamma)$, $\varepsilon>0$ and $s_1, \cdots, s_n\in \Gamma$ as above,
$$|\tau_1(a)-\tau_2(a)|=\Big|\tau_2\big(\frac{1}{n}\sum_{1}^{n} \pi_{s_i}a\pi_{s_i}^*-\tau_1(a)1\big)\Big|\leq \Big\|\frac{1}{n}\sum_{1}^{n} \pi_{s_i}a\pi_{s_i}^*-\tau_1(a)1\Big\|<\varepsilon,$$
thus $\tau_1=\tau_2$. 

If moreover $C^*_\pi(\Gamma)$ has a faithful trace $\tau$ and $I\unlhd C^*_\pi(\Gamma)$ is a non zero closed ideal, then for any non zero element $a_0\in I$, $\alpha:=\tau(a_0)\neq 0$, and for $a:\frac{1}{\alpha}a_0\in I$, given $\varepsilon>0$, the identity $1\in C^*_\pi(\Gamma)$ is within $\varepsilon$ of the element $\frac{1}{n}\sum_{1}^{n} \pi_{s_i}a\pi_{s_i}^*\in I$, for some choice of $s_i$'s in $\Gamma$, thus $1\in I$, as  $I$ is closed, that is, $I=C^*_\pi(\Gamma)$. 

For the converse, we adapt an argument of Kennedy in \cite[Theorem 3.8]{ke}: assume by the way of contradiction that $\pi$ does not have Powers averaging property, i.e.,  there is  $a\in C^*_\pi(\Gamma)$ such that,  for the weak$^*$-closed convex hull $L$ of the set $\{\pi_sa\pi_s^*: s\in\Gamma\}$, we have $\tau(a)1\notin L$, for a trace $\tau\in\mathcal T(C^*_\pi(\Gamma))$. By Hahn-Banach theorem, there is a non zero bounded functional $\varphi\in C^*_\pi(\Gamma)^{*}$ such that 
$$\inf_{b\in L}\big|\varphi(b)-\tau(a)\varphi(1)\big|>0.$$
Put $K:=\overline{{\rm conv}}(\Gamma\varphi)$. This is a convex affine $\Gamma$-space in $C^*_\pi(\Gamma)^{*}$, and we may choose a minimal affine $\Gamma$-space $K_0$ inside $K$ containing $\varphi$. Look at the Jordan decomposition $\varphi=(\varphi_1-\varphi_2)+i(\varphi_3-\varphi_4)$, with $\varphi_i=\alpha_i\omega_i$, for positive scalars $\alpha_i$ and states $\omega_i$, for $i=1,\cdots, 4.$ As $\varphi$ is non zero, so is one of coefficients $\alpha_i$, say $\alpha_1\neq 0$. By Corollary \ref{conv}, $\alpha_1\tau\in \overline{{\rm conv}}(\Gamma\varphi_1)$. Choose a net $(\theta_j)_{j\in J}$ of finitely supported maps $\theta_j: \Gamma\to \mathbb R^{+}$ with $\sum_{s\in\Gamma} \theta_j(s)=1,$ for each $j\in J$, such that 
$$\sum_{s\in\Gamma} \theta_j(s)(s\cdot\varphi_1)\to\alpha_1\tau,$$
in the weak$^*$-topology. Passing to a subnet, we may assume that  
$$\sum_{s\in\Gamma} \theta_j(s)(s\cdot\varphi_i)\to\beta_i\tau_i,$$
in the weak$^*$-topology, for some positive scalars $\beta_i$ and states $\tau_i$, for $i=2,3,4$, i.e., 
$(\alpha_1\tau-\beta_2\tau_2)+i(\beta_3\tau_3-\beta_4\tau_4)\in K_0$. Since $\tau$ is $\Gamma$-invariant, we may repeat this argument one more time to get $(\alpha_1\tau-\beta_2\tau)+i(\beta^{'}_3\tau^{'}_3-\beta^{'}_4\tau^{'}_4)\in K_0$, for positive scalars $\beta^{'}_i$ and  states $\tau^{'}_i$, for $i=3,4.$  Doing this two more times, $K_0$ shall contain a scalar multiple $\alpha\tau$ of $\tau$. But by minimality, this forces $K_0=\{\alpha\tau\}$, that means, $\varphi=\alpha\tau$ (and in particular, $\alpha\neq 0)$.  back to the Hahn-Banach separation inequality, we get 
$$\inf_{b\in L}\big|\alpha\tau(b)-\alpha\tau(a)\big|=\alpha\inf_{b\in L}\big|\tau(b)-\tau(a)\big|>0,$$
which is absurd, as $\tau(\pi_sa\pi_s^*)=\tau(a)$, for each $s\in \Gamma$, and so by linrearity and continuity, $\tau(b)=\tau(a)$, for each $b\in L$.\qed 	
\end{proof}

\begin{lemma} \label{fix}
	If there is a $\Gamma$-map $\psi: \ell^\infty(\Gamma)\to \mathcal B_{\pi}$, then $\Gamma_\omega$ is amenable for each $\omega\in\mathcal P(\mathcal B_\pi)$. In this case, if moreover $\mathcal B_{\pi}$ is an AW$^*$-algebra, then there is a $\Gamma$-map $\phi: C^*_\pi(\Gamma)\to \mathcal B_{\pi}$ satisfying $\phi(\pi_s)=1_{{\rm Fix}(s)}$, where {\rm Fix}$(s)=\{\omega\in\mathcal{P}(\mathcal{B}_{\pi}): s\omega=\omega\}$. 
\end{lemma}
\begin{proof}
Compose $\psi$ with the canonical embedding $\ell^\infty(\Gamma_\omega)\hookrightarrow \ell^\infty(\Gamma)$ to get a $\Gamma$-map $\phi: \ell^\infty(\Gamma_\omega)\to \mathcal B_{\pi}$. Let us observe that $\omega\circ \phi$ is $\Gamma_\omega$-left translation invariant:
$$\langle \omega\circ\phi, s\cdot f\rangle=\langle s^{-1}(\omega\circ\phi), f\rangle=\langle (s^{-1}\omega)\circ\phi, f\rangle=\langle \omega\circ\phi, f\rangle,$$
for $s\in \Gamma$, $f\in \ell^\infty(\Gamma_\omega)$. 

For the second statement, given $\omega\in\mathcal P(\mathcal B_\pi)$, by Lamma \ref{can}$(iii)$, $1_{\Gamma_\omega}$ extends to a state $\rho_\omega\in\mathcal S(C^*_\pi(\Gamma))$ and $\omega\mapsto \rho_\omega$ is a $\Gamma$-map: $\mathcal P(\mathcal B_\pi)\to\mathcal  S(C^*_\pi(\Gamma))$, whose range $X$ is a $\Gamma$-space. We claim that indeed $X$ is a $\Gamma$-boundary. Since $\mathcal P(\mathcal B_\pi)$ is extremely disconnected by assumption, Fix$(s)$ is clopen by the Frol\'{i}k theorem \cite{f}. In particular, if $\omega_i\to\omega$ in the weak$^*$-topology of $\mathcal P(\mathcal B_\pi)$, $\rho_{\omega_i}\to\rho_{\omega}$ in weak$^*$-topology of $\mathcal  S(C^*_\pi(\Gamma))$; i.e., the above map is continuous, and $X$ is a $\Gamma$-boundary, as claimed.\qed
\end{proof}

\begin{corollary}
Assume that $\lambda\preceq\pi$ and $\mathcal B_{\pi}$ is an AW$^*$-algebra with identity element 1. If $C^*_\pi(\Gamma)$ has unique trace property with trace $\tau_0$ and $\tau_0$ composed with the canonical inclusion $\mathbb C\hookrightarrow \mathcal B_{\pi}; \ z\mapsto z1$ is the only $\Gamma$-map: $C^*_\pi(\Gamma)\to \mathcal B_{\pi}$, then the action of $\Gamma$ on $\mathcal B_{\pi}$ is free.	
\end{corollary}
\begin{proof}
We may choose the trace $\tau_0$ as in Lemma \ref{can}. For the map $\phi$ as in Lemma \ref{fix}, by the uniqueness of $\Gamma$-maps, $\phi(\pi_s)=\tau_0(\pi_s)1. $ Thus, $$\tau_0(\pi_s)=\langle \tau_0(\pi_s)1,\omega\rangle=\langle \phi(\pi_s),\omega\rangle=\langle 1_{{\rm Fix}(s)},\omega\rangle, \ \ (\omega\in\mathcal P(\mathcal B_\pi)).$$
Since $\tau_0(\pi_s)=\delta_{se},$ we have Fix$(s)$ is equal to $\mathcal P(\mathcal B_\pi)$ for $s=e$, and  empty otherwise, that is, the action $\Gamma\curvearrowright \mathcal P(\mathcal B_\pi)$ is free.\qed	
	\end{proof}
\section{Simplicity}\label{sim}

Let $\pi: \Gamma\to B(H_\pi)$ be a unitary representation and $\Lambda\unlhd \Gamma$ be a  normal subgroup, then 
$$H^\Lambda_\pi:=\{\xi\in H_\pi: \pi_t\xi=\xi\ \ (t\in\Lambda)\}$$
is a closed Hilbert subspace and by normality of $\Lambda$,   $\pi_sH^\Lambda_\pi\subseteq H^\Lambda_\pi,$ for each $s\in \Gamma$. In particular, we have a sub-representation 
$$\pi^\Lambda: \Gamma\to B(H^\Lambda_\pi); \ \ \pi^\Lambda(t)\xi:=\pi_t\xi \ \ (t\in\Gamma, \xi\in H^\Lambda_\pi)$$
which imitates  the construction of  the quasi-regular representation from  left regular representation (c.f., \cite{bkko}), which is known to be weakly regular \cite{k}. Recall that a  unitary representation is called weakly regular if it is weakly contained in the left regular representation \cite[page 64]{bkko}.

\begin{definition}
	Given a unitary representation $\pi$, we say that $\Gamma$ is $C^*_\pi$-simple if $C^*_\pi(\Gamma)$ is simple.
\end{definition}

\begin{proposition} \label{rad5}
	The following are equivalent:
	
	$(i)$ $\Gamma$ is $C^*_\pi$-simple,
	
	$(ii)$ every representation $\sigma$ weakly contained in $\pi$ is weakly equivalent to $\pi$. 
\end{proposition}
\begin{proof}
	$(i)\Rightarrow (ii)$.  For unitary representation $\pi$, by the universality of $C^*_{\rm max}(\Gamma)$, there is a morphism of C*-algebras: $C^*_{\rm max}(\Gamma)\to C^*_\pi(\Gamma)$. Following \cite{dl}, we denote the kernel of this morphism by C*Ker$(\pi)$. If $C^*_\pi(\Gamma)$ is simple, C*Ker$(\pi)$ is maximal
	among closed two-sided ideals of $C^*_{\rm max}(\Gamma)$. Now if  $\sigma\preceq\pi$ then C*Ker$(\pi)\subseteq$C*Ker$(\sigma)$ \cite[Theorem 7]{dl}, and so by maximality, C*Ker$(\pi)=$C*Ker$(\sigma)$. Thus $\sigma$ is weakly equivalent to $\pi$ by applying again the quoted result above. 
	
	$(ii)\Rightarrow (i)$. The assumption means that C*Ker$(\pi)$ is maximal
	among closed two-sided ideals of the form C*Ker$(\sigma)$. But  closed two-sided ideals of  $C^*_{\rm max}(\Gamma)$ are of the form C*Ker$(\sigma)$, for some unitary representation $\sigma$, thus   C*Ker$(\pi)$ is a maximal
	 ideal of $C^*_{\rm max}(\Gamma)$,  consequently  $C^*_\pi(\Gamma)$ is simple.
	\qed
\end{proof}

Recall that the boundary $\mathcal{B}_{\pi}$ is C*-embeddable if there is a *-homomorphic copy of
$\mathcal{B}_{\pi}$ in $B(H_\pi)$. In this case,  $\tilde{\mathcal{B}}_{\pi}$ denotes the C*-algebra generated by $\mathcal{B}_{\pi}\cup \pi(\Gamma)$
in $B(H_\pi)$. This is a unital $C^*$-algebra whose unit is $\pi(e)=$id$_{H_\pi}$. 

\begin{theorem}\label{emb}
	If $\mathcal{B}_{\pi}$ is  C*-embeddable, the simplicity of  $C^*_\pi(\Gamma)$ is equivalent to simplicity of $\tilde{\mathcal{B}}_{\pi}$.
\end{theorem}
\begin{proof}
	Suppose that $C^*_\pi(\Gamma)$ is simple. Given a closed ideal $I$ of $\tilde{\mathcal{B}}_{\pi}$, Let $J:=I\cap C^*_\pi(\Gamma)$, then $J=0$ or $C^*_\pi(\Gamma)$. If $J=C^*_\pi(\Gamma)$, then $I\supseteq C^*_\pi(\Gamma)$ contains $\pi(e)$, which is the identity element of $\tilde{\mathcal{B}}_{\pi}$, thus $I=\tilde{\mathcal{B}}_{\pi}$. On the other hand, if $J=0$, then $I\cap   C^*_\pi(\Gamma)=0$, and 
	$$C^*_\pi(\Gamma)\subseteq \tilde{\mathcal{B}}_{\pi}\subseteq I(C^*_\pi(G)),$$
	by \cite[Proposition 3.20]{bek}, thus $I=0$ by \cite[Lemma 1.2]{h1}. 
	
	Conversely, let $\tilde{\mathcal{B}}_{\pi}$ be simple. Let $J$ be a closed  ideal of  $C^*_\pi(\Gamma)$  and let
	 $q: C^*_\pi(\Gamma)\to C^*_\pi(\Gamma)/J$ be the corresponding quotient map. Then $q$ is $\Gamma$-equivariant with respect to the actions of $\Gamma$ on its domain and range by Ad$_\pi$ and Ad$_{q\circ\pi}$. By $\Gamma$-injectivity, $q$ 	extends to a $\Gamma$-equivariant map $\tilde q: \tilde{\mathcal{B}}_{\pi}\to I_\Gamma(C^*_\pi(\Gamma)/J)$ whose multiplicative domain contains $C^*_\pi(\Gamma)$. In particular, the $\Gamma$-action on $I_\Gamma(C^*_\pi(\Gamma)/J)$ is also implemented by
	Ad$_{q\circ\pi}$. Similar to the situation in the proof of \cite[Theorem 6.2]{kk}, the problem is that $\tilde q$ is not a homomorphism.
	
	Applying $\Gamma$-injectivity again, there is a $\Gamma$-equivariant projection
	$$\phi: I_\Gamma(C^*_\pi(\Gamma)/J)\to \tilde q(\mathcal B_\pi),$$
	which induces a Choi–Effros product on $\tilde q(B_\pi),$ making it C*-algebra isomorphic to $B_\pi$, with respect to the Choi-Effros product on $\mathcal B_\pi$ coming from $\Gamma$-equivariant projection: $B(H_\pi)\to \mathcal B_\pi$. Let $\rho: \tilde q(\mathcal B_\pi)\to \mathcal B_\pi$ be the composition of the restriction of $\phi$ to $\tilde q(\mathcal B_\pi)$ with the above mentioned isomorphism. 
	
	Next, the $\Gamma$-action on $ I_\Gamma(C^*_\pi(\Gamma)/J)$ canonically extends to a $\Gamma$-action
	on $ I_\Gamma(C^*_\pi(\Gamma)/J)^{**}$ such that the inclusion
	$\tilde q(\mathcal B_\pi)^{**}\subseteq I_\Gamma(C^*_\pi(\Gamma)/J)^{**}$ 
	is $\Gamma$-equivariant and we have  a normal  $\Gamma$-equivariant surjective $*$-homomorphism
	$$\rho^{**}: \tilde q(\mathcal B_\pi)^{**}\to \mathcal B_\pi^{**}.$$ 
	Then $B_\pi^{**}\cong c(\rho^{**}) q(\mathcal B_\pi)^{**}$, where $c(\rho^{**})$ is the support projection of $\rho^{**}$. Since $c(\rho^{**})\in I_\Gamma(C^*_\pi(\Gamma)/J)^{**}$, we may consider the map $$\sigma: \tilde{\mathcal{B}}_{\pi}\to I_\Gamma(C^*_\pi(\Gamma)/J)^{**};\ \ x\mapsto \tilde q(x)c(\rho^{**}).$$
	Since ker$(\rho^{**})$ 
	is $\Gamma$-invariant, $c(\rho^{**})$ commutes with $C^*_\pi(\Gamma)/J$. Hence the restriction of $\sigma$ to $C^*_\pi(\Gamma)$ is a $*$-homomorphism. On the other hand, the restriction of $\sigma$ to $\mathcal B_\pi$ is a $*$-isomorphism
	between $c(\rho^{**}) q(\mathcal B_\pi)$ and $\mathcal B_\pi$. Since $\tilde{\mathcal{B}}_{\pi}$ is the C*-algebra generated by $\mathcal{B}_{\pi}\cup \pi(\Gamma)$
	in $B(H_\pi)$, it follows that $\sigma$ is a $*$-homomorphism.
	Now by simplicity of $\tilde{\mathcal{B}}_{\pi}$, $\sigma$  is injective, hence $J=$ker$(q)=0$.\qed
\end{proof}

\begin{theorem}\label{tf}
	
If $\pi\preceq\lambda$ and	$\tilde{\mathcal{B}}_{\pi}$ is simple, then the action of  $\Gamma$ on ${\mathcal{B}}_{\pi}$ is topologically free. The converse also holds when $\Gamma$ is an exact group and $\mathcal{B}_{\pi}$ is C*-embeddable.
\end{theorem}
\begin{proof}
Let $\omega$ be a state of the C*-algebra ${\mathcal{B}}_{\pi}$ with Choi-Effros product. Let
$$\Gamma_\omega:=\{s\in\Gamma: s\omega=\omega\}$$
where $s\omega(x)=\omega(s^{-1}\cdot x)$, for $x\in {\mathcal{B}}_{\pi}$. Since $\Gamma$ acts by Ad$_\pi$ on ${\mathcal{B}}_{\pi}$ and the identity of the unital C*-algebra ${\mathcal{B}}_{\pi}$ is the identity operator on $H_\pi$, $s\omega$ is again a state. It follows from the properties of action that $\Gamma_\omega\leq\Gamma$ is a subgroup. Let $\Lambda:=\bigcap \Gamma_{s\omega}$, where $s$ runs over $\Gamma$. This is nonempty as it contains the identity of the group. Also, since $t\Gamma_\omega t^{-1}\subseteq \Gamma_{t\omega}$, $\Lambda$ is a normal subgroup of $\Gamma$. Let $\pi^\Lambda$ be the corresponding sub-representation of $\pi$ as described in the first paragraph of the current section and $p^\Lambda: H_\pi\to H_\pi^\Lambda$ be the corresponding orthogonal projection. Since $p^\Lambda$ commutes with the range of $\pi^\Lambda$, $(\iota, \pi^\Lambda, H_\pi^\Lambda)$ is a covariant pair, where $\iota$ is an embedding $\iota_0$ of $\mathcal{B}_{\pi}$  in $B(H_\pi)$  (which exists by assumption) composed with the canonical surjection 
$$B(H_\pi)\twoheadrightarrow B(H_\pi^\Lambda); \ x\mapsto p^\Lambda xp^\Lambda.$$
By universal property of the full crossed product, the integrated representation $\pi^\Lambda\rtimes \iota$ extends to a $*$-epimomorphism: $\mathcal{B}_{\pi}\rtimes \Gamma\twoheadrightarrow
p^\Lambda \tilde{\mathcal{B}}_{\pi}p^\Lambda$. If $I\unlhd \mathcal{B}_{\pi}\rtimes \Gamma$ is a closed ideal with $I\cap \mathcal{B}_{\pi}=\{0\}$, then $J:=\pi^\Lambda\rtimes \iota(I)$ is a closed ideal of $p^\Lambda \tilde{\mathcal{B}}_{\pi}p^\Lambda$. The latter C*-algebra is simple as a corner of a simple C*-algebra, thus $J=0$ or $p^\Lambda \tilde{\mathcal{B}}_{\pi}p^\Lambda$. But $\pi^\Lambda\rtimes \iota$ acts as identity on ${\mathcal{B}}_{\pi}$ and so the second option for $J$ cannot hold, that is, $J=0$, i.e., $I\subseteq \ker(\pi^\Lambda\rtimes \iota)$. 

On the other hand, 
$$\pi^\Lambda\preceq{\rm ind}_\Lambda^\Gamma(\pi^\Lambda|_\Lambda)={\rm ind}_\Lambda^\Gamma(1_\Lambda)=\lambda,$$ hence, $\pi^\Lambda\rtimes \iota\preceq\tilde\lambda$ for the canonical surjection $\tilde\lambda: \mathcal{B}_{\pi}\rtimes\Gamma\twoheadrightarrow\mathcal{B}_{\pi}\rtimes_r\Gamma$. Therefore, 
$$I\subseteq \ker(\pi^\Lambda\rtimes \iota)\subseteq \ker\tilde\lambda.$$
Summing up, we have shown that for each closed ideal $I\unlhd \mathcal{B}_{\pi}\rtimes \Gamma$  with $I\cap \mathcal{B}_{\pi}=\{0\}$,
we have  $I\subseteq \ker\tilde\lambda.$ Now since $\pi\preceq\lambda$, $\mathcal{B}_{\pi}$ is abelian by \cite[Theorem 3.21]{bek}, and so it follows from \cite[Theorem 2]{as} that the action of $\Gamma$ on ${\mathcal{B}}_{\pi}$ is topologically free.

For the converse, assume moreover that $\Gamma$ is an exact group and $\mathcal{B}_{\pi}$ is C*-embeddable. Then $\mathcal{B}_{\pi}\rtimes\Gamma=\mathcal{B}_{\pi}\rtimes_r\Gamma$ by \cite[Theorem 5.3]{d}. Since $\mathcal{B}_{\pi}$ is C*-embeddable,
there is a $*$-epimomorphism:  $\mathcal{B}_{\pi}\rtimes_r\Gamma \to \tilde{\mathcal{B}}_{\pi}$,  mapping $C^*_\lambda(\Gamma)$ to $C^*_\pi(\Gamma)$ canonically. In particular, $\pi\preceq\lambda$, and so $\mathcal{B}_{\pi}$ is abelian. Now it follows from \cite[Corollary 1]{as} that $\mathcal{B}_{\pi}\rtimes_r\Gamma$ is simple, and so is $\tilde{\mathcal{B}}_{\pi}$. \qed
\end{proof}

\begin{remark} $(i)$ In the proof of the above theorem, we constructed a normal $\Lambda$. It is worth noting that if $\Gamma/\Lambda$ is amenable, then the condition $\pi\preceq\lambda$ is automatic, as by \cite[1.F.21]{bd},
	$$\pi\preceq{\rm ind}_\Lambda^\Gamma(\pi^\Lambda|_\Lambda)={\rm ind}_\Lambda^\Gamma(1_\Lambda)=\lambda.$$
By the same argument, the condition is automatic if $\Lambda$ is coamenable (c.f., the proof of \cite[Theorem 3.28]{bek}). 
Also in the first part of the proof, the condition $\pi\preceq\lambda$ could be replaced by the apparently weaker (c.f., \cite[Theorem 3.21]{bek}) condition that  $\mathcal{B}_{\pi}$ is commutative.

$(ii)$ In the proof of the above lemma, we used the action of $\Gamma$ on the state space of  $\mathcal{B}_{\pi}$. Since the action preserves positivity, it sends pure states to pure states. In particular, the same argument could be phrased based on pure state space or equivalently on spectrum of $\mathcal{B}_{\pi}$ (the latter being compatible with the approach of  Archbold-Spielberg in defining freeness of such actions \cite{as}).

$(iii)$ If $s\in\Lambda$ and $x\in {\mathcal{B}}_{\pi}$, then by normality of $\Lambda$,  $t\omega(s\cdot x)=t\omega(x)$, for each state of the form $t\omega$ on ${\mathcal{B}}_{\pi}$. It follows from Proposition \ref{mp} that $s\cdot x=x$, i.e., $\Lambda$ acts trivially on ${\mathcal{B}}_{\pi}$, therefore $\Lambda$ is  $\pi$-amenable by \cite[Theorem 4.4]{bek}.

$(iv)$ Finally note that in the above proof, the restricted representation $\pi|_{\Gamma_\omega}$ is known to be an amenable representation of $\Gamma_\omega$ \cite[Lemma 3.14]{bek}. 
\end{remark}

\begin{lemma}
	
	If  $\mathcal{B}_{\pi}$ is C*-embeddable, then the action of  $\Gamma$ on ${\mathcal{B}}_{\pi}$ is topologically free if and only if it is free. 
\end{lemma}
\begin{proof}
If  $\mathcal{B}_{\pi}$ is C*-embeddable, then since by construction there is a surjective c.p. projection: $B(H_\pi)\to \mathcal{B}_{\pi}$, it follows that $\mathcal{B}_{\pi}$ is an injective C*-algebra, and so an $AW^*$-algebra \cite[IV.2.1.7]{bl}. In particular, its spectrum is extremely disconnected. Now the action of  $\Gamma$ on ${\mathcal{B}}_{\pi}$ is topologically free if and only if the same holds for the action of  $\Gamma$ on the spectrum of ${\mathcal{B}}_{\pi}$ \cite[Definition 1]{as} and we use the same convention for free actions. Now the result follows from \cite[Lemma 3.4]{bkko}.	\qed
\end{proof}

We use \cite[Definition 6.1]{bkko} (see also, \cite[Definition 1.1]{bfs}).

\begin{definition}
A subgroup $\Lambda\leq\Gamma$ is called normalish if for every $n\geq 1$ and $t_1,\cdots,t_n\in\Gamma$, the intersection $t_1\Lambda t^{-1}_1\cap\cdots \cap t_n\Lambda t^{-1}_n$ is an infinite set. 
\end{definition}

In the proof of Theorem \ref{tf}, we constructed subgroups $\Gamma_\omega$, for a (pure) state $\omega$   of the unital C*-algebra ${\mathcal{B}}_{\pi}$ (with Choi-Effros product). In the case of the left regular representation, among these are the stabilizer subgroups associated to the action of $\Gamma$ on the Furstenberg boundary $\partial_F\Gamma$, which are known to be amenable \cite[Lemma 2.7]{bkko}. Here we need to work with $\pi$-amenable subgroups, and the subgroups $\Gamma_\omega$ are not necessarily $\pi$-amenable. However, if the restricted representation $\pi|_{\Gamma_\omega}$ is amenable in the sense of Bekka \cite{b}, then $\Gamma_\omega$ is $\pi$-amenable by \cite[Proposition 4.2]{bek}. 

\begin{definition} \label{pmf} 
	The action of $\Gamma$ on a (unital) C*-algebra $A$ is {\it proximal} if the corresponding action on the state space $\mathcal S(A)$ is proximal, i.e., given states $\omega_1,\omega_2\in\mathcal S(A)$, there is a net $(t_i)$ in $\Gamma$ with $\lim_i t_i\omega_1= \lim_i t_i\omega_2$. It is {\it minimal} if for each $\omega\in \mathcal P(A)$ in the pure state space, the orbit $\{s\omega: s\in\Gamma\}$ is dense in $\mathcal P(A)$. Finally, it is {\it topologically free} if the set $\{\omega\in\mathcal P(A): s\omega\neq \omega\}$ is weak$^*$-dense in $\mathcal P(A)$. 
\end{definition}

\begin{remark} \label{tf2} 
	$(i)$ When $A=C_0(X)$ is commutative, $\mathcal S(A)=P(X)$ is the set of probability measures on $X$, and a given  action of $\Gamma$ on $X$ is strongly proximal in the sense of Furstenberg \cite{f2} iff the induced action on $C_0(X)$ is proximal in the above sense and it is topologically transitive iff the induced action on $C_0(X)$ is minimal in the above sense. 
	
	$(ii)$ One might wonder if using state space $\mathcal S(A)$ or pure state space $\mathcal P(A)$ instead of spectrum $\hat A$ in the above definition (as well as in the proof of Theorem \ref{tf}) is a deviance from the approach of Archbold-Spielberg in defining (topological) freeness in terms of the action on the spectrum \cite[Definition 1]{as} (rather than the state space). However, when $A$ is unital, for $s\in \Gamma$ if ${\rm Fix}(s):=\{\sigma\in \hat A: s\sigma=\sigma\}$ and $\mathcal P^s(A):=\{\omega\in \mathcal P(A): s\omega=\omega\}$ are the corresponding fix-spaces, then the canonical  map: $\mathcal P(A)\to \hat A$  maps  $\mathcal P^s(A)$ onto ${\rm Fix}(s)$, since if $\omega\in \mathcal P(A)$ and $\sigma:=\pi_\omega\in \hat A$ is the corresponding GNS-representation with cyclic vector $\xi_\omega$, then
	$$\langle \pi_{s\omega}(a)\xi_{s\omega}, \xi_{s\omega}\rangle=s\omega(a)=\omega(s^{-1}\cdot a)=\langle \pi_{\omega}(s^{-1}\cdot a)\xi_{\omega}, \xi_{\omega}\rangle=\langle s\pi_{\omega}(a)\xi_{\omega}, \xi_{\omega}\rangle,$$
	for each $a\in A$, and so $\pi_{s\omega}=s\pi_{\omega}$ in $\hat A$ by  \cite[2.4.1]{di}. Now since the above canonical map is continuous open \cite[3.4.11]{di}, ${\rm Fix}(s)$ has nonempty interior in $\hat A$ iff $\mathcal P^s(A)$ has nonempty interior in $\mathcal P(A)$. This means that the topological freeness in the above sense is equivalent to topological freeness in the sense of \cite[Definition 1]{as}, at least when $\hat A$ is Hausdorff (see, \cite[Remark$(i)$]{as}).
	
	$(iii)$ Recall that  an automorphism $\alpha$ of a C*-algebra A is called {\it properly outer} if for any
	non-zero invariant ideal $I$ of $A$ and any inner automorphism $\beta$ of $I$, $\|\alpha|_J-\beta\|= 2$ \cite{e}. An action $\alpha$ of $\Gamma$ on $A$ is properly outer if $\alpha_t$ is properly outer, for each $t\neq e$.  If $\Gamma\curvearrowright A$ is topologically
	free then it is properly outer:  Suppose on the contrary that $\alpha_s$ is not properly outer for some $s\neq e$. Let $I$ be a non zero ideal of $A$
	invariant under $\alpha_s$, then $s\sigma=\sigma$, for each $\sigma\in \hat I$ \cite[Proposition 1]{as}. Regard $\mathcal P(I)$ as an open (non-empty) subset of $\mathcal P(A)$ \cite[2.11.8]{di}, and observe that $\mathcal P(I)\subseteq {\rm Fix}(s)$, by $(ii)$. In particular,  $\mathcal P^s(A)$ has non-empty interior in $\mathcal P(A)$, thus   $\Gamma\curvearrowright A$ is not topologically
	free (note that in this argument we do not require $\hat A$ to be Hausdorff). 
\end{remark}

\begin{proposition} \label{mp}
The action of $\Gamma$ on $\mathcal{B}_{\pi}$ is minimal.  If there is a $\Gamma$-embedding:  $\ell^\infty(\Gamma)\hookrightarrow B(H_\pi)$, then this action is also proximal.	
\end{proposition}
\begin{proof}
Let $F$ be the closure in $\hat{ \mathcal{B}}_{\pi}$  of the set $\{\pi_{s\omega}: s\in \Gamma\}$ consisting of GNS-constructions of pure states in the orbit of $\omega$, and let  $I:=\ker(F)$ be the corresponding closed ideal of $\mathcal{B}_{\pi}$. For the normal subgroup $\Lambda$ as in the proof of Theorem \ref{tf}, we have a sub-representation $\pi^\Lambda\leq \pi$, for which we get a  $\Gamma$-map: $\mathcal{B}_{\pi}\to\mathcal{B}_{\pi^\Lambda}$ \cite[Proposition 3.18]{bek}, with kernel $I$. Since $H_{\pi^\Lambda}$ is a $\Gamma$-invariant closed subspace of $H_\pi$, there is a $\Gamma$-embedding: $\mathcal{B}_{\pi^\Lambda}\hookrightarrow B(H_\pi)$, i.e., the above $\Gamma$-map: $\mathcal{B}_{\pi}\to\mathcal{B}_{\pi^\Lambda}$ could be regarded as a $\Gamma$-map into $B(H_\pi)$,  which is then isometric by $\pi$-essentiality \cite[Proposition 3.4]{bek}, and so $I=\ker(F)=0$, thus $F= \hat{ \mathcal{B}}_{\pi}$. It now follows that $\{s\omega: s\in\Gamma\}$ is dense in $\mathcal P(\mathcal{B}_{\pi})$.

Given $\varrho\in \mathcal S(\mathcal{B}_{\pi})$, let $K$ be the weak$^{*}$-closed convex hull of the set $\{s\varrho: s\in\Gamma\}$. We claim that $\mathcal P(\mathcal{B}_{\pi})\subseteq K$. Otherwise,  choosing pure state $\omega\notin K$, by Hahn-Banach, one may choose $x\in \mathcal{B}_{\pi}$ and $\varepsilon>0$ with
$$\langle s\varrho, x\rangle\leq  \langle \omega, x\rangle-\varepsilon\leq\|x\|-\varepsilon\ \ (s\in \Gamma).$$
Since $\ell^\infty(\Gamma)$ $\Gamma$-embeds in $B(H_\pi)$, the  $\Gamma$ map: $$\mathcal{B}_{\pi}\to \ell^\infty(\Gamma);\ \ x\mapsto(s\mapsto\langle s\varrho, x\rangle)$$ could be regarded as a $\Gamma$-map into $B(H_\pi)$, and so is isometric by $\pi$-essentiality, contradicting the above inequalities, proving the claim. Now by Milman's theorem \cite{m}, $\omega$ is in the weak$^{*}$-closure of the set $\{s\varrho: s\in\Gamma\}$. Next, given states $\varrho_1,\varrho_2\in\mathcal S(\mathcal{B}_{\pi})$, there is a net $(t_i)$ in $\Gamma$ with $\lim_i t_i(\frac{1}{2}(\varrho_1+ \varrho_2))=\omega$. Since  $\mathcal{B}_{\pi}$ is unital, passing to subnets, we may assume that the nets $t_i\varrho_k$ converge to states $\omega_k$, for $k=1,2$. Thus, $\omega=\frac{1}{2}(\omega_1+ \omega_2)$, and by extremality, $\omega=\omega_1=\omega_2$, that is,  $\lim_i t_i\varrho_1=\lim_i t_i \varrho_2$.     \qed	
\end{proof}

In the next two results, $\Gamma_\omega$ is the stabilizer subgroup of a state $\omega\in\mathcal S(\mathcal{B}_{\pi})$ as  in the proof of Theorem \ref{tf}.

\begin{lemma} \label{nish}
	 If $\Gamma$ has no non-trivial finite normal subgroup and the action of $\Gamma$ on $\mathcal{B}_{\pi}$ is not topologically free, then the subgroups $\Gamma_\omega$ are normalish for each pure state $\omega\in\mathcal P(\mathcal{B}_{\pi})$.
\end{lemma}
\begin{proof}
	Since the action of $\Gamma$ on $\mathcal{B}_{\pi}$ is not topologically free, by Remark \ref{tf2}, there is $s\neq e$ such that ${\rm Fix}(s)$ has non-empty interior in $\mathcal P(\mathcal{B}_{\pi})$.	By strong proximality any finite set  states $\{\omega_1,...,\omega_n\}$ in $\mathcal S(\mathcal{B}_{\pi})$ can be mapped into the
	interior of $Fix(s)$ by some $t_0\in \Gamma$, that is $(t_0^{-1}st_0)\omega_i=\omega_i$, for $1\leq i\leq n$. Now given $\omega\in\mathcal S(\mathcal{B}_{\pi})$ and a finite subset $\{t_1,...,t_n\}$ of $\Gamma$, the above observation applied to
	$\omega_i := t_i\omega$ shows that the intersection $t_1\Gamma_\omega t^{-1}_1\cap\cdots \cap t_n\Gamma_\omega t^{-1}_n$ is non-trivial. Arguing as in the proof of \cite[Theorem 6.2]{bkko}, if this is finite for some
	choice of $t_i$’s, but non-trivial for all choices of $t_i$’s, then
	$\bigcap_{t\in G} t\Gamma_\omega t^{-1}$ is a non-trivial finite
	normal subgroup of $\Gamma$, which cannot exist by assumption. Hence $\Gamma_\omega$ is normalish as claimed. \qed	
\end{proof}

The next result, which follows immediately from Lemma \ref{nish}, extends \cite[Theorem 6.2]{bkko}.

\begin{corollary} \label{nish2}
	Assume that  $\Gamma_\omega$ is $\pi$-amenable for some pure state $\omega\in\mathcal P(\mathcal{B}_{\pi})$. If $\Gamma$ has no non-trivial finite normal subgroup and no
	$\pi$-amenable normalish subgroup, then  the action of $\Gamma$ on $\mathcal{B}_{\pi}$ is  topologically free.
\end{corollary}

Let $\Gamma$ act on a unital C*-algebra $A$. We say that $A$ has only {\it countably many stabilizers} if the set $\{\Gamma_\omega: \omega\in \mathcal{P}(A)\}$ of stabilizer subgroups is at most countable. We adapt the proof of \cite[Theorem 6.12]{bkko} to get the following result. 

\begin{corollary} \label{nish3}
	Assume that $\mathcal{B}_{\pi}$ has only countably many stabilizers. If $\Gamma$ has no non-trivial finite normal subgroup and the action of $\Gamma$ on $\mathcal{B}_{\pi}$ is faithful then it is also topologically free.
\end{corollary}
\begin{proof}
If the action of $\Gamma$ on $\mathcal{B}_{\pi}$ is not topologically free, by Remark \ref{tf2}, there is $s\neq e$ such that ${\rm Fix}(s)$ has non-empty interior in $\mathcal P(\mathcal{B}_{\pi})$. Let us observe that there is $\omega_0\in {\rm Fix}(s)$ such that stabilizer $\Gamma_{\omega_0}$ fixes a non-empty open
subset of $\mathcal P(\mathcal{B}_{\pi})$ pointwise.
For $\Lambda\leq \Gamma$ and $F\subset\mathcal P(\mathcal{B}_{\pi})$  let $\ker(\Lambda):=\bigcap_{s\in \Lambda} {\rm Fix}(s)$ and ${\rm hull}(F):= \bigcup_{\omega\in F}\ker(\Gamma_\omega)$. Note that for each pure state $\omega$, ${\rm hull}(\{\omega\})$ is a closed subset of $\mathcal P(\mathcal{B}_{\pi})$  containing $\omega$. By  assumption that $\Gamma$ contains only countably many stabilizers, there is a countable subset $F\subseteq {\rm Fix}(s)$  such that
${\rm hull}({\rm Fix}(s))={\rm hull}(F)$. On the other hand,  ${\rm hull}({\rm Fix}(s))$ contains ${\rm Fix}(s)$ and thus has non-empty interior. In particular, by the Baire category theorem, there is $\omega_0\in F$ such that ${\rm hull}(\{\omega_0\})$ has non-empty
interior, say $U$. Then $\Gamma_{\omega_0}$ fixes $U$ pointwise, as claimed.

Next, by minimality, there are $s_1,\cdots,s_n\in\Gamma$ such
that $(s_1 \cdot U)\cup\cdots \cup (s_n\cdot  U)= \mathcal P(\mathcal{B}_{\pi})$. Since the elements in $s_k\Gamma_{\omega_0}s_k^{-1}$ fix every point in $s_k\cdot U$, the elements in the intersection $\bigcap_{k=1}^{n}s_k\Gamma_{\omega_0}s_k^{-1}$ fix every point in $\mathcal P(\mathcal{B}_{\pi})$. Since $\Gamma_{\omega_0}$ is
normalish by Lemma \ref{nish},  the action could not be faithful. \qed	
\end{proof}

We mimic \cite[Definition 8.1]{bkko} as follows.

\begin{definition}
We say that a unitary representation $\pi$ of $\Gamma$ has Connes-Sullivan property $(CS)$ if for every representation $\sigma$ weakly contained in $\pi$, there exists an SOT-neighborhood $U$
	of the identity in $B(H_\sigma)$ such that $\sigma^{-1}(U) \subseteq {\rm Rad}_\pi(\Gamma)$.
\end{definition}

The group $\Gamma$ itself is said to have Connes-Sullivan property $(CS)$ whenever its left regular representation $\lambda$ has property $(CS)$ in the above sense (c.f., \cite[Definition 8.1]{bkko}).

It is known that if  ${\rm Rad}_\pi(\Gamma)$ is  trivial, then $\Gamma$ has no non-trivial $\pi$-amenable normal subgroup \cite[Corollary 4.5]{bek}. We show that in the presence of the Connes-Sullivan property $(CS)$, normal could be replaced by normalish when  $\lambda\precsim\pi$. 

\begin{lemma} \label{cs}
	If $\pi$  has Connes-Sullivan property $(CS)$ and ${\rm Rad}_\pi(\Gamma)$ is  trivial, then  for each $\sigma\preceq\pi$, $\sigma(\Gamma)$ is discrete in the relative SOT-topology of $B(H_\sigma)$. 
\end{lemma}
\begin{proof}
	By $(CS)$ we may choose  an SOT-neighborhood $U$
	of the identity $I_\sigma$ in $B(H_\sigma)$ such that $\sigma^{-1}(U) \subseteq {\rm Rad}_\pi(G)=\{e\}$. Therefore, $\sigma^{-1}(U)=\{e\}$, thus $U\cap \sigma(\Gamma)=\sigma\sigma^{-1}(U)=\sigma\{e\}=\{I_\sigma\}$, that is, $\{I_\sigma\}$ is open inside $\sigma(\Gamma)$ in the relative SOT-topology, as claimed. \qed	
\end{proof}

The next lemma is proved as part of the argument of the proof of \cite[Proposition 8.2]{bkko}.

\begin{lemma} \label{nish4}
	If  $\Lambda\leq \Gamma$ is normalish,  $\lambda_{\Gamma/\Lambda}(\Gamma)$ is not discrete in the relative SOT-topology of $B(\ell^2(\Gamma/\Lambda))$. 
\end{lemma}

The next result extends \cite[Proposition 8.2]{bkko}. 

\begin{corollary} \label{nish5}
	If $\lambda\preceq\pi$, $\pi$  has Connes-Sullivan property $(CS)$ and ${\rm Rad}_\pi(\Gamma)$ is  trivial, then $\Gamma$ has no $\pi$-amenable normalish subgroup.
\end{corollary}
\begin{proof}
Let $\Lambda\leq \Gamma$. Since $\lambda\preceq\pi$, there is a $\Gamma$-inclusion $\pi(\Lambda)^{'}\hookrightarrow \lambda(\Lambda)^{'}$, and so $\pi$-amenability of $\Lambda$ implies its $\lambda$-amenability, which  then implies amenability by \cite[Proposition 4.3]{bek} (there it is assumed that $\Lambda$ is a normal subgroup, but this is not used in the proof). Next, recall that when $\Lambda$ is amenable, $\lambda_{\Gamma/\Lambda}\preceq\lambda$.  By assumption, $\lambda_{\Gamma/\Lambda}$ is equivalent to some sub-representation $\sigma$ of $\pi$, say via Ad$_{u}$, for some unitary $u$ between the corresponding Hilbert spaces. Then Ad$_{u}$ restricts to a homeomorphism between $\lambda_{\Gamma/\Lambda}(\Gamma)$ and $\sigma(\Gamma)$, and so $\lambda_{\Gamma/\Lambda}(\Gamma)$ is not discrete by Lemma \ref{cs}. Therefore, $\Lambda$ is not normalish by Lemma \ref{nish4}.  \qed	
\end{proof}

\begin{remark}
In the above corollary,  if we drop the assumption $\lambda\preceq\pi$, we may still conclude that $\Gamma$ has no amenable normalish subgroup in $\mathcal W(\pi)$, where  $\mathcal W(\pi):= \{\Lambda\leq \Gamma: \lambda_{\Gamma/\Lambda} \preceq\pi\}$. Since $\mathcal W(\lambda)$ consists of the set
of  amenable subgroups of $\Gamma$, this is still more general than the earlier results on absence of amenable normalish subgroups. Note that,  $\mathcal W(\lambda_{\Gamma/\Lambda})$ consists of subgroups
$H$ with $\lambda\leq \Sigma\leq \Gamma$ such that $\Lambda$ is coamenable in $\Sigma$ \cite[Proposition 2.3]{ks}. Also note that if $\mathcal W(\pi)$ admits a $\Gamma$-invariant probability, then
$C^*_\pi(\Gamma)$ admits a trace \cite[Remark 3.19]{ks}.
\end{remark}

The next result now follows from Lemma \ref{nish}.

\begin{corollary} \label{tf3}
	If $\Gamma$ has no non-trivial finite normal subgroup, $\lambda\precsim\pi$, $\pi$  has Connes-Sullivan property $(CS)$ and ${\rm Rad}_\pi(\Gamma)$ is  trivial, and $\Gamma_\omega$ is amenable for some pure state $\omega\in\mathcal P(\mathcal{B}_{\pi})$, then the action of $\Gamma$ on $\mathcal{B}_{\pi}$ is  topologically free.
\end{corollary}

Next we turn into Power's averaging property. First we need a series of lemmas.

\begin{lemma} \label{char}
	If  $\Lambda\leq \Gamma$ is subgroup such that $\lambda_{\Gamma/\Lambda}\preceq\pi$, then the characteristic function $1_\Lambda$ extends to a state on $C^*_\pi(\Gamma)$. In particular, this holds when $\Lambda$ is amenable.
\end{lemma}
\begin{proof}
First note that there is a state 	
\end{proof}

\section{Concrete Examples} \label{ex}

Up to here we have worked with an abstract representation $\pi$. In this section we examine the results of previous sections for concrete instances of $\pi$.  

\subsection{{\bf Representations coming from boundary actions}} The properties of representations associated to boundary actions are recently studied in \cite{ks}. An action of $\gamma$ on a compact Hausdorff space $X$ is strongly proximal if for every probability measure $\nu\in P(X)$ the weak$^*$-closure
of the orbit $\Gamma\nu$ contains some point measure $\delta_x$, $x\in X$. A  {\it boundary action} is an action which is both minimal and strongly
proximal. A related notion is that of  an {\it extreme boundary action}, which is characterized by the condition that given a closed subset
$C\neq X$ and open subset $U\neq \emptyset$, there is $t\in \Gamma$  such that $t\cdot C\subseteq U$. This condition is known to be stronger than being a boundary action
\cite[Theorem 2.3]{gl}.

For a $\Gamma$-space $X$ and $x\in X$, the stabilizer subgroup of $x$ is $$\Gamma_x := \{t\in\Gamma:
	t\cdot x = x\}.$$ The open stabilizer at $x$ is the subgroup
$$\Gamma^0_x := \{t\in\Gamma: t \ {\rm fixes\ an\ open\ neighborhood\ of\ }x\}.$$ Then $\Gamma^0_x\unlhd\Gamma_x$, for $x\in X$.  The open stabilizer
map Stab$^0$ from $X$ to the topological space Sub$(\Gamma)$ of subgroups of $\Gamma$, endowed with the Chabauty
topology (the restriction of the product topology on $\{0, 1\}^\Gamma$) is defined by $x\mapsto \Gamma^0_x$. We say that  the action $\Gamma\curvearrowright X$
has {\it Hausdorff germs} if the open stabilizer
map Stab$^0$ is continuous on $X$ \cite[page 4]{ks} (c.f., \cite[Definition 2.9]{lb}). Similarly one could define the stabilizer
map Stab$:X\to$ Sub$(\Gamma)$; $x\mapsto \Gamma_x$. This map is continuous at points $x$ in which $\Gamma_x=\Gamma^0_x$ \cite[Lemma 2.2]{lb}. 

An {\it action class} of $\Gamma$ is a collection $\mathfrak C$ of compact $\Gamma$-spaces. A representation $\pi\in$Rep$(\Gamma)$ is a $\mathfrak C$-representation, writing $\pi\in$Rep$_{\mathfrak{C}}(\Gamma)$, if there is a $\Gamma$-map: $C(X) \to B(H_\pi)$, for some $X\in\mathfrak{C}$. As a concrete example, if $X\in\mathfrak{C}$ and $\nu$ is a quasi-invariant $\sigma$-finite measure on $X$, then the Koopman representation $\kappa_\nu$ defined by $$\kappa_\nu(t)\xi(x):=
[\frac{dt\nu}{d\nu}]^{\frac{1}{2}}(x)\xi(t^{-1}\cdot x);\ \ t\in\Gamma, \xi\in L^2(X, \nu),$$ is a $\mathfrak{C}$-representation \cite[Example 3.3]{ks}. If 
$\pi$ is in Rep$_{\mathfrak{C}}(\Gamma)$, then so is $\sigma$, for each $\sigma\preceq\pi$ \cite[Proposition 3.4]{ks}. For $\pi\in$Rep$_{\mathfrak{C}}(\Gamma)$,
$\mathfrak{C}^\pi$ is the collection of all $X\in \mathfrak{C}$ such that there is a $\Gamma$-map from $C(X)$
to $B(H_\pi)$. 

A subgroup $\Lambda\leq \Gamma$  is a $\mathfrak C$-subgroup, writing $\Lambda\in$Sub$_{\mathfrak{C}}(\Gamma)$, 
if $\lambda_{\Gamma/\Lambda}\in$Sub$_{\mathfrak{C}}(\Gamma)$. This holds iff  $\Lambda$ fixes a probability on some $X\in \mathfrak{C}$ \cite[Proposition 3.10]{ks}. The subgroup class Sub$_{\mathfrak{C}}(\Gamma)$
is invariant under conjugation by an element $s\in G$, since representations  $\lambda_{\Gamma/\Lambda}$ and $\lambda_{\Gamma/\Lambda^s}$ are clearly unitarily equivalent (cf., \cite[Remark 6.2]{bk}).

Alternatively one could say that  a normal subgroup $\Lambda\unlhd \Gamma$, is called a $\mathfrak C$-subgroup, writing
if   
$\pi^\Lambda\in$Sub$_{\mathfrak{C}}(\Gamma)$, where $\pi^\Lambda$ is the sub-representation of $\pi$ constructed in the beginning of Section \ref{s}. The next lemma shows that the two notions coincide for normal subgroups.

\begin{lemma} \label{eq}
Let $\mathfrak C$ be an action class. For a normal subgroup $\Lambda\unlhd \Gamma$, the following are equivalent:

$(i)$ $\lambda_{\Gamma/\Lambda}\in$Sub$_{\mathfrak{C}}(\Gamma)$,

$(ii)$ $\Lambda$ fixes a probability on some $X\in \mathfrak{C}$,

$(ii)$ $\pi^\Lambda\in$Sub$_{\mathfrak{C}}(\Gamma)$.
\end{lemma}
\begin{proof}
$(i)\Leftrightarrow (ii)$. This is proved in \cite[Proposition 3.10]{ks}.

$(iii)\Rightarrow (ii)$. Let $\psi: C(X)\to B(H_\pi^\Lambda)$ be a $\Gamma$-map for $X\in\mathfrak{C}$. Since $\Lambda$ acts trivially on $B(H_\pi^\Lambda)$, the composition of $\psi$ with any state on $B(H_\pi^\Lambda)$ would be a $\Lambda$-equivariant state on $C(X)$.

$(ii)\Rightarrow (iii)$. Let $\nu$ be a $\lambda$-invariant probability on $X\in \mathfrak{C}$, then we have the chain of $\Gamma$-maps
$$C(X)\to \ell^\infty(\Gamma/\Lambda)=\ell^\infty(\Gamma)^\Lambda\hookrightarrow\ell^\infty(\Gamma)\hookrightarrow B(\ell^2(\Gamma)),$$ 
where the first map is the Poisson map. Now since
$$\pi^\Lambda\preceq{\rm ind}_\Lambda^\Gamma(\pi^\Lambda|_\Lambda)={\rm ind}_\Lambda^\Gamma(1_\Lambda)=\lambda,$$ 
there is a $\Gamma$-map:  $B(\ell^2(\Gamma))\to B(H_\pi^\Lambda)$, which is composed with the above chain of $\Gamma$-maps to give a $\Gamma$-map:  $C(X)\to B(H_\pi^\Lambda)$.	\qed	
\end{proof}
	
In \cite{ks}, the authors specify three action classes, i.e., the action classes $\mathfrak{B}_0$  of all $\Gamma$-boundaries, and sub-classes $\mathfrak{B}_1$ and $\mathfrak{B}_2$ of respectively faithful and topologically free $\Gamma$-boundaries. The subgroup class Sub${}_{\mathfrak{B}_2}(\Gamma)$ coincides with the class  of weakly parabolic subgroups in the sense of Bekka and Kalantar \cite[Definition 6.1]{bk}. The subgroup classes Sub${}_{\mathfrak{B}_1}(\Gamma)$ and Sub${}_{\mathfrak{B}_2}(\Gamma)$ are either empty or contain all amenable subgroups of $\Gamma$. Moreover, Sub${}_{\mathfrak{B}_1}(\Gamma)$ is non-empty iff $\Gamma$ has trivial amenable
radical, and Sub${}_{\mathfrak{B}_2}(\Gamma)$ is non-empty iff $\Gamma$ is C*-simple (see, \cite{kk}, \cite{bkko}). For subgroups $\Lambda_1\leq \Lambda_2$, with $\Lambda_1$ co-amenable
in (to) $\Lambda_2$ (relative to $\Gamma$; cf., \cite[7.C]{cm}), then $\Lambda_1\in$Sub${}_{\mathfrak{B}_k}(\Gamma)$ implies $\Lambda_2\in$Sub${}_{\mathfrak{B}_k}(\Gamma)$, for $k=1,2$ (because if $\Lambda_1$ fixes a probability on some $X\in \mathfrak{B}_k$, so does $\Lambda_2$). Finally, Sub${}_{\mathfrak{B}_1}(\Gamma)$ contains no non-trivial normal subgroup of $\Gamma$ \cite[Proposition 3.15]{ks} and Sub${}_{\mathfrak{B}_2}(\Gamma)$ contains no recurrent subgroup \cite[Corollary 6.8]{bk} (see also, \cite[Theorem 1.1]{ke}).

In Proposition \ref{rad3}, we observed that all extendable traces on $C^*_\pi(\Gamma)$ are supported on {\rm Rad}$_{\pi}(\Gamma)=\ker(\Gamma\curvearrowright \mathcal B_\pi)$. In general, all traces on $C^*_\pi(\Gamma)$ are supported on
$N_\pi:=\bigcap_{X\in \mathfrak{B}_0^\pi}\ker(\Gamma\curvearrowright X)$ \cite[Corollary 3.6]{ks}. It is natural to ask the relation between {\rm Rad}$_{\pi}(\Gamma)$ and $N_\pi$.

\begin{remark}
	Since there is a {\rm Rad}$_{\pi}(\Gamma)$-invariant state on $B(H_\pi)$, there is such a state on $C(X)$, for $X\in \mathfrak{B}_0^\pi$. But {\rm Rad}$_{\pi}(\Gamma)\unlhd\Gamma$, and so by strong proximality, {\rm Rad}$_{\pi}(\Gamma)$ acts trivially on such $X$, that is, {\rm Rad}$_{\pi}(\Gamma)\subseteq N_\pi$ \cite[Remark 3.7]{ks}. Conversely, if $\mathcal B_\pi$ is commutative, the inverse Gelfand transform is a $\Gamma$-map: $C(\hat{\mathcal B}_\pi)\to \mathcal B_\pi$. Composing this with the incluson map: $\mathcal B_\pi\hookrightarrow B(H_\pi)$, we get a $\Gamma$-map: $C(\hat{\mathcal B}_\pi)\to B(H_\pi)$. But $\hat{\mathcal B}_\pi$ is a $\Gamma$-boundary by Proposition \ref{mp}, thus  $\hat{\mathcal B}_\pi\in\mathfrak{B}_0^\pi$, therefore, in this case, {\rm Rad}$_{\pi}(\Gamma)=N_\pi$. 
\end{remark}

The group $\Gamma$ acts (continuously) on Sub$(\Gamma)$
by conjugation. An {\it invariant random subgroup} (IRS) on $\Gamma$ is a $\Gamma$-invariant regular probability
measure on Sub$(\Gamma)$. It is known that $\Gamma$ has unique trace property iff $\Gamma$ has no non-trivial amenable IRS \cite[Corollary
1.5]{bdl}, \cite[Theorem 1.3]{bkko}). It is not however known if every IRS with support contained in Sub${}_{\mathfrak{B}_1}(\Gamma)$ is trivial. \cite[Question 3.16]{ks}. Recall that  $\mathcal W(\pi) := \{\Lambda\leq\Gamma: \lambda_{\Gamma/\Lambda}\precsim\pi\}$. This is  a
closed conjugation-invariant subset of Sub$(\Gamma)$ and a compact $\Gamma$-space. If $\mathcal W(\pi)$ admits a $\Gamma$-invariant probability, then
$C^*_\pi(\Gamma)$ admits a trace (cf., the proof of \cite[Theorem 3.17]{ks}).  The converse  is not true: for any unital C*-algebra
$A$ admitting a trace, a group $\Gamma$ of unitaries generating $A$ and
containing $-1_A$ gives a unitary representation  which admits trace not come from an IRS \cite[Remark 3.19]{ks}. The set $\mathcal W(\lambda)$ consists of all amenable subgroups of $\Gamma$. In general, $\mathcal W(\lambda_{\Gamma/\Lambda})$ contains all subgroups
containing $\Lambda$ co-amenably \cite[Proposition 2.3]{ks}. It is known that any IRS supported on $\mathcal W(\pi)$ is supported on $N_\pi$. In particular, if $\Lambda\in$Sub${}_{\mathfrak{B}_1}(\Gamma)$, any IRS supported on $\mathcal W(\lambda_{\Gamma/\Lambda})$ is trivial.

In Corollary \ref{rad4}, we observed that if $\Gamma$ acts faithfully on $\mathcal{B}_{\pi}$, then  $C^*_\pi(\Gamma)$ has either has no extendable trace or has a unique one. If we specify the action class of $\pi$, stronger results are possible \cite[Corollaries 3.8, 3.12]{ks}.

\begin{example}\label{ex1}
$(i)$ If $\pi\in$Rep$_{\mathfrak{B_1}}(\Gamma)$, then either $C^*_\pi(\Gamma)$ admits no trace, or
otherwise $\lambda\precsim\pi$ and the canonical trace is the unique trace
on $C^*_\pi(\Gamma)$.

$(ii)$ If $\pi\in$Rep$_{\mathfrak{B_2}}(\Gamma)$, then the canonical trace is the unique boundary map on $C^*_\pi(\Gamma)$,  $\lambda\precsim\pi$, and  $C^*_\pi(\Gamma)$ has a maximal proper ideal.

$(iii)$ Let $\Lambda\leq \Gamma$ be a subgroup for which $C^*_{\lambda_{\Gamma/\Lambda}}(\Gamma)$ is nuclear.
Then $C^*_{\lambda_{\Gamma/\Lambda}}(\Gamma)$ admits a trace iff $1_\Gamma\precsim\lambda_{\Gamma/\Lambda}.$ Indeed, if $X$ is a compact $\Gamma$-space and $\pi\in$Rep$_{\{X\}}(\Gamma)$ is such that $C^*_\pi(\Gamma)$ is
nuclear and admits a trace, then $X$ admits a $\Gamma$-invariant probability \cite[Proposition 3.11]{ks}. Now the above assertion follows from the  special case $X=\beta({\Gamma/\Lambda})$.  

$(iv)$ If $X$ is a $\Gamma$-boundary and $N:=\ker(\Gamma \curvearrowright X)$, then $C^*_\pi(\Gamma)$ admits a trace iff the induced action $\Gamma/N\curvearrowright X$ is topologically free, iff $\lambda_{\Gamma/N} \precsim\pi$. In this case, $1_N$ extends to the unique trace on
$C^*_\pi(\Gamma)$. Indeed, if $\tau$ is a trace on $C^*_\pi(\Gamma)$, then $x\mapsto \tau(a)1_{\partial_F \Gamma}$
is a boundary map on $C^*_\pi(\Gamma)$, which is then has to be unique. Thus  $\Gamma/N\curvearrowright X$ has to be topologically free (see, \cite[Theorem 4.8]{ks}). Conversely, if $\Gamma/N$ acts topologically free on $X$, then
the restriction of  boundary map on
$C^*_\pi(\Gamma)$ to $\pi(\Gamma)$ is $1_N\cdot 1_{\partial_F \Gamma}$. Composing this boundary map with a point mass at $z\in \partial_F \Gamma$, 
we get a state on $C^*_\pi(\Gamma)$ whose restriction to to $\pi(\Gamma)$ is $1_N$. But $\delta_N\in \ell^2(\Gamma/N)$ is a
cyclic vector for $\lambda_{\Gamma/N}$, thus $\lambda_{\Gamma/N} \precsim\pi$. 
\end{example}

As observed in part $(ii)$ of the above example, that uniqueness of the boundary maps is related to uniqueness of trace and C*-simplicity. Recall that for a $\Gamma$-C*-algebra $A$, a boundary map on $A$ is a $\Gamma$-map: $A \to  C(\partial_F \Gamma)$. 

To see the way these notions are related, suppose that $\Gamma$ is C*-simple and $x\in \partial_F \Gamma$. Then by $\Gamma$-injectivity, $\Gamma_x$ is
amenable,  $C^*_{\lambda_{\Gamma/\Gamma_x}}(\Gamma)=C^*_{\lambda}(\Gamma)$, and the unique boundary map on $C^*_{\lambda_{\Gamma/\Gamma_x}}(\Gamma)$ coincides with the canonical trace on $C^*_{\lambda}(\Gamma)$. In particular, the fixed point set $(\partial_F \Gamma)^s$ has empty interior, for $s\neq e$, i.e., $\partial_F \Gamma$ is a topologically free $\Gamma$-boundary. Conversely, if $\Gamma$ admits  a topologically free boundary, then by  part $(i)$ of the above example, the canonical trace is the unique boundary map on $C^*_{\lambda}(\Gamma)$, which is also
faithful. Hence $C^*_{\lambda}(\Gamma)$ is simple by \cite[Proposition 3.1]{ks} (see, \cite[Corollary 4.5]{ks}, \cite[Theorem 6.2]{kk}).  A similar argument shows that $\Gamma$ is C*-simple iff the canonical
trace is the unique boundary map on $C^*_{\lambda}(\Gamma)$ (see, \cite[Corollary 4.6]{ks}, \cite[Theorem 3.4]{ke}).

The above argument extends to arbitrary topologically free $\Gamma$-boundaries: If $\Gamma$ is C*-simple and  $X$ is
a $\Gamma$-boundary with some amenable point stabilizer $\Gamma_x$, then
$X$ is topologically free.
Indeed, since $\lambda_{\Gamma/\Gamma_x}$ is unitarily equivalent to $\lambda$, the canonical trace
is the unique boundary map map on $C^*_{\lambda_{\Gamma/\Gamma_x}}(\Gamma)$, and so the fixed point set $X^s:=\{x\in X: s\cdot x=x\}$ has empty interior, for $s\neq e$ \cite[Corollary 4.7]{ks}, \cite[Proposition 1.9]{bkko}).  

\subsection{{\bf Groupoid representations}}

Let $X$ be a compact $\Gamma$-space. A {\it groupoid representation} of $(\Gamma, X)$
is a nondegenerate covariant representation $(\pi, \rho)$ of $(\Gamma, C(X))$ such that $\pi_s\rho(f) = \rho(f)$, for $s\in \Gamma,  f\in C(X)$, with ${\rm supp}(f)\subseteq X^s$ (or equivalently, with ${\rm supp}(f)$ contained in the interior of $X^s$).

An example of  a groupoid representation is the covariant pair $(\kappa_\nu, \rho)$, where $\nu$ is a $\sigma$-finite quasi-invariant measure on $X$,  $\kappa_\nu$ is
the Koopman representation of $\Gamma$ on $L^2(X, \nu)$, and $\rho$ is the multiplication representation of $C(X)$ in $L^2(X, \nu)$. 
Another example is the covariant pair $(\lambda_{\Gamma/\Lambda}, P_x)$, where $x\in X$ satisfies $\Gamma_x^0\leq \Lambda\leq \Gamma_x$ and $P_x : C(X) \to B(\ell^2(\Gamma/\Lambda))$ is the Poisson map \cite[Proposition 4.3]{ks}. 

When $X$ is a $\Gamma$-boundary with a groupoid representation $(\pi, \rho)$ of
$(\Gamma, X)$, there exists a  $\Gamma$-map$:  C^*_\pi(\Gamma)\to C(X)$ iff the action  $\Gamma\curvearrowright X$
has Hausdorff germs \cite[Proposition 4.14]{ks}. Indeed, if the action  $\Gamma\curvearrowright X$
has Hausdorff germs the interior $(X^s)^{o}$ is clopen, and the map $\pi(s)\mapsto 1_{(X^s)^{o}}$ maps $\pi(\Gamma)$ inside a copy of $C(X)$ in $C(\partial_F \Gamma)$, which extends to a $\Gamma$-map on $C^*_\pi(\Gamma)$. 

If $(\pi, \rho)$ is a groupoid representation of $(\Gamma, X)$ with integrated representation $\rho\rtimes\pi$ of the maximal crossed product $C(X)\rtimes\Gamma$, then there is a unique boundary map  on the C*-algebra  $C(X)\rtimes_{\rho\rtimes\pi}\Gamma$ generated by the range of $\rho\rtimes\pi$, whose restriction to $C^*_\pi(\Gamma)$  is the
unique boundary map on $C^*_\pi(\Gamma)$. Indeed, by $\Gamma$-injectivity of $C(\partial_F \Gamma)$, there exists a boundary map: $C(X)\rtimes_{\rho\rtimes\pi}\Gamma\to C(\partial_F \Gamma)$. Since any boundary map on $C^*_\pi(\Gamma)$ extends to a boundary map
on $C(X)\rtimes_{\rho\rtimes\pi}\Gamma$, uniqueness of the above map implies uniqueness of its restriction on $C^*_\pi(\Gamma)$ (cf., \cite[Theorem 4.4]{ks}). When this boundary map  on  $C(X)\rtimes_{\rho\rtimes\pi}\Gamma$ is also faithful, then so is its restriction on $C^*_\pi(\Gamma)$. Now since every ideal of the latter is included in the left kernel of a boundary map \cite[Proposition 3.1]{ks}, $C^*_\pi(\Gamma)$ would be  simple in this case. A similar argument based on left kernels shows that  the unique boundary map on $C(X)\rtimes_{\rho\rtimes\pi}\Gamma$ is faithful exactly when $C(X)\rtimes_{\rho\rtimes\pi}\Gamma$ is simple. 

\begin{example}\label{ex2}
	$(i)$ If $X$ is a faithful $\Gamma$-boundary and $(\pi, \rho)$ is a groupoid representation
	of $(\Gamma, X)$, then it follows from Example \ref{ex1}$(iv)$ that $C^*_\pi(\Gamma)$  admits a trace if and only if $X$ is topologically free. In this case, $\lambda\precsim\pi$ and the canonical trace is the unique trace
	on $C^*_\pi(\Gamma)$ (cf., \cite[Corollary 4.9]{ks}).
	
	$(ii)$ As a particular case of $(i)$, if $X$ is a $\Gamma$-boundary, $X$ is topologically free iff $\lambda\precsim \kappa_\nu$, for the Koopman representation $\kappa_\nu$ of some quasi-invariant measure $\nu$ on $X$ (see, \cite [Corollary 4.10]{ks}, \cite[Theorem 31]{r}). 
	
	$(iii)$ If $X$ is a faithful $\Gamma$-boundary,  $N:=\ker(\Gamma \curvearrowright X)$,  and $(\pi, \rho)$ is a groupoid representation
	of $(\Gamma, X)$, then either $C^*_\pi(\Gamma)$
	has no trace or it has a unique trace and a maximal (indeed maximum) closed ideal $I=I_{\rm max}$ such that $C^*_\pi(\Gamma)/I\simeq C^*_{\lambda_{\Gamma/N}}$ is simple. indeed, if $C^*_\pi(\Gamma)$ admits a trace, then the trace has to be unique by Example \ref{ex1}$(iv)$. By the same argument as in 
	Example \ref{ex1}$(iv)$, one could observe that if $\sigma\in$Rep$(\Gamma)$ and $\sigma\precsim\pi$, then $\lambda_{\Gamma/N}\precsim\sigma$. This shows the existence of the maximal closed ideal $I_{\rm max}$.  Finally, since 
	$\Gamma/N$ acts topologically
	free on $X$, boundary action, hence $C^*_{\lambda_{\Gamma/N}}$ is simple by the argument in the second paragraph after Example \ref{ex1} (cf., \cite[Theorem 6.2]{kk}).
	
	$(iv)$ In general, existence of a unique boundary map does not imply simplicity \cite[Remark 4.12]{ks}: For $\Gamma:= PSL_n(\mathbb Z)$, $n\geq 3$, the projective space $\mathbb P(\mathbb R^n)=:X$ is a topologically free $\Gamma$-boundary
	and there is $x\in X$ with $\Gamma_x$ non-amenable. Thus by part $(i)$,  the unique boundary map on $C^*_{\lambda_{\Gamma/\Gamma_x}}$ is the canonical trace, but 
	$C^*_{\lambda_{\Gamma/\Gamma_x}}$  could not be simple, since otherwise then we would have that $\lambda_{\Gamma/\Gamma_x}$ would be unitarily equivalent to $\lambda$, which contradicts the non-amenability of $\Gamma_x$. 
	
	$(v)$ For a compact $\Gamma$-space $X$, following Le Boudec
	and Matte Bon \cite{lb}, we denote the set of continuity points of the stabilizer map Stab$: X\to$ Sub$(\Gamma)$ by $X_0$. We have  $X_0 = \{x \in X : \Gamma_x = \Gamma_x^0\}$ \cite[Lemma 2.2]{lb}. Similarly, the set of continuity points of the open stabilizer map Stab$^0: X\to$ Sub$(\Gamma)$ is denoted by $X_0^0$. Clearly, $X_0\subseteq X_0^0$. 
	When $X$ is a $\Gamma$-boundary with a groupoid representation $(\pi, \rho)$ of
	$(\Gamma, X)$, for each  $x\in X_0^0$, $\lambda_{\Gamma/\Gamma_x^0}\precsim\pi$. Indeed, the composition of a boundary map on $C^*_\pi(\Gamma)$ with point mass at a point $y\in \partial_F \Gamma$ is a state  which restricts to $1_{\Gamma_x^0}$ on $\Gamma$. Now the claim follows as 
	$\delta_{\Gamma_x^0}\in \ell^2(\Gamma/\Gamma_x)$
	is a cyclic vector for $C^*_{\lambda_{\Gamma/\Gamma_x^0}}(\Gamma)$. The same argument works for $\pi$ replaced by any $\sigma\precsim\pi$ \cite[Theorem 5.2]{ks}. In particular, $C^*_{\lambda_{\Gamma/\Gamma_x^0}}(\Gamma)$ is simple for 	 $x\in X_0^0$. The  simplicity of $C^*_{\lambda_{\Gamma/\Gamma_x}}(\Gamma)$ for $x\in\partial_F\Gamma$, observed by Kawabe \cite[Corollary 8.5]{ka}, is a special case of the above observation, since $\Gamma_x=\Gamma_x^0$ for $x\in\partial_F\Gamma$ \cite[Lemma 3.4]{bkko}.
	
	$(v)$ For a compact $\Gamma$-space $X$ and $x\in X$, if $C^*_{\lambda_{\Gamma/\Gamma_x}}(\Gamma)$ is simple then $\lambda_{\Gamma/\Gamma_x}\precsim\lambda_{\Gamma/\Gamma_x^0}$ by \cite[Lemma 5.5]{ks}, that is, $\Gamma/\Gamma_x^0$ is coamenable in $\Gamma/\Gamma_x$. In particular, $\Gamma/\Gamma_x^0$ is amenable. The same argument done backwards shows that if $\Gamma/\Gamma_x^0$ is amenable then $C^*_{\lambda_{\Gamma/\Gamma_x}}(\Gamma)$ is simple.
	
	$(vi)$ An argument similar to that used in part $(iv)$ shows that for  a compact $\Gamma$-boundary $X$ with a  groupoid representation $(\pi, \rho)$, $P_x\rtimes\lambda_{\Gamma/\Lambda} \precsim\rho\rtimes\pi$, when $x\in X_0^0$. In this case,  $C(X)\rtimes_{P_x\rtimes\lambda_{\Gamma/\Lambda}}\Gamma$ is simple, where $P_x : C(X) \to B(\ell^2(\Gamma/\Lambda))$ is the Poisson map (cf., \cite[Corollary 5.11]{ks}). 
	
	$(vii)$ The Thompson’s group $V$ is the set of piecewise linear bijections on the interval [0, 1) which are right continuous, have finitely many points of non-differentiability, all
	being dyadic rationals, and have a derivative which is a power of 2 at each point of
	differentiability. The Thompson’s group $T$ consists of those elements of $V$ which have at
	most one point of discontinuity, and Thompson’s group $F$ consists of those elements
	of $V$ which are homeomorphisms of [0, 1), or equivalently, those elements of $T$ which vanish at 0. Haagerup and Olesen proved that if $T$ is C*-simple, then $F$ has to be non-amenable \cite{ho}. The converse is now known to be true \cite{lb}. Let $Y$ be the set obtained from reals $\mathbb R$ by replacing each $y\in \mathbb Z[1/2]$ by two elements $y_{-}, y_{+},$ endowed  with the order topology. Then $K := Y \cap [0_{+}, 1_{-}]$
	is a Cantor set and $V$ on $K$ by
	$$g\cdot s := g(s),\ g\cdot y_{+}:= g(y)_{+}, \ g\cdot y_{-} := g_{-}(y_{-})_{-},$$
	for $g\in V, s\in (0, 1) \backslash \mathbb Z[1/2]$, and $y\in \mathbb Z[1/2]$, where $g_{-}(y_{-})$ is the left limit of $g$ at $y_{-}$. Now for $\Gamma:=T$, $X=K$,  and $x=0_{+}$, we have $\Gamma_x=F$, and $X$ is a faithful topologically non-free $\Gamma$-boundary, and so Example \ref{ex1}$(iv)$,  $C^*_{\lambda_{T/F}}(T)$
	does not admit traces. Also $\Gamma_x^0= \{g\in F: g'(0) = 1\}$ and  $F/T^0_{0_{+}}\to \mathbb Z$ maps $g$ with $g'(0) = 2^a$ to $a$, and so is an isomorphism. Also, the interior of $g\cdot K$ is clopen by \cite[Proposition 6.1]{ks}. Therefore, 
	$C^*_{\lambda_{T/F}}(T)$ is simple by part $(v)$ (cf., \cite[Theorem 6.2]{ks}). Similarly, one can observe that for the $V$-invariant set $X := [0, 1) \cap \mathbb  Z[1/2]$ and the corrsponding unitary representation $\pi$ of $V$ on $\ell^2(X)$, $C^*_\pi(\Gamma)$  is simple and admits no traces. Finally, identifying $\mathbb S^1$ with [0, 1), we get a boundary action of $T$ such that
	the open stabilizer of 0 is the commutator subgroup $[F,F]=:F_2$ \cite[Theorem 4.1]{cfp}. Since $F_2$ is coamenable in $F$, $\lambda_{T/F}\precsim\lambda_{T/F_2}$, but one can show that these two representations are not unitarily equivalent \cite[Example 6.5]{ks}, therefore, $C^*_{\lambda_{T/F_2}}(T)$ is not simple.
\end{example}

More general results along the line of Example \ref{ex2}$(iv)$ could be obtained. Let $X$ be a compact $\Gamma$-space. Let $\partial(\Gamma, X)$ be  the spectrum of the $\Gamma$-injective envelope $I_{\Gamma}(C(X))$, and $b_X : \partial(\Gamma, X)\to  X$ be the
canonical continuous equivariant map induced by the embedding of $C(X)$ in its $\Gamma$-injective envelope. When $X$ is a singleton, $\partial(\Gamma, X)=\partial_F\Gamma$ is the
Furstenberg boundary  of $\Gamma$ \cite[Theorem 3.11]{kk}. When the action $\Gamma \curvearrowright X$ is minimal, there is a unique equivariant conditional expectation: $C(X)\rtimes \Gamma\to C(X)$ restricting to the canonical trace on $C^*_r(\Gamma)$ iff the action $\Gamma \curvearrowright \partial(\Gamma, X)$ is  faithful iff no stabilizer $\Gamma_x$  contain a non-trivial amenable normal subgroup iff some open stabilizer  $\Gamma_x^0$ contains no non-trivial amenable normal subgroup \cite[Theorem 4.6]{ks2}.

\begin{example}\label{ex3}
	
	$(i)$ If $X$ is a minimal compact $\Gamma$-space and $\Gamma x^0$ is amenable for
	some $x\in X$, then the action $\Gamma \curvearrowright \partial(\Gamma, X)$ is  faithful iff  $\Gamma \curvearrowright X$ is  faithful. Indeed, if $\Gamma \curvearrowright X$ is not faithful its kernel is a non-trivial amenable normal subgroup of $\Gamma$, thus the action $\Gamma \curvearrowright \partial(\Gamma, X)$ is not faithful by the above result. The converse implication is immediate. 
	
	$(ii)$ If $X$ is a minimal locally compact $\Gamma$-space, $x\in X$ and
	$y\in b_X^{-1}(x)$, then $$ \Lambda:= \{s \in\Gamma: y \ {\rm is \ in \ the \ closure \ of \ } b_X^{-1}\big((X^s)^{o}\big)\}$$ is a subgroup with $\Gamma_x^0\leq\Lambda\leq\Gamma_x$ and  every groupoid representation is weakly contained in 
	$(\lambda_{\Gamma/\Lambda}, P_x)$, and so . In particular, $C_0(X)\rtimes_{P_x\lambda_{\Gamma/\Lambda}}\Gamma$ is simple. In particular,  $C_0(X)\rtimes_{P_x\rtimes\lambda_{\Gamma/\Gamma_x^0}}\Gamma$ is simple for $x\in X_0^0$ (cf., \cite[Proposition 5.1]{ks2}). This extends Example \ref{ex2}$(vi)$. 

\end{example}

\section{Non commutative boundaries} \label{ncbdry}

In this section we extends the results of previous section for the $\Gamma$-boundary $X$ replaced by a (unital) $C^*$-algebra $A$ with a $\Gamma$-action. 

A (unital) $C^*$-algebra $A$ is called a $\Gamma$-$C^*$-{\it algebra} if there is an action $\Gamma \curvearrowright A$ by $*$-automorphisms. We call a linear
map $\phi: A \to B$ between $\Gamma$-$C^*$-algebras is a $\Gamma$-{\it map} if it is unital completely positive
(u.c.p.), and $\Gamma$-equivariant.

A $\Gamma$-$C^*$-algebra $A$ is called $\Gamma$-{\it injective} if, given a completely isometric $\Gamma$-map
$\psi : B \to C$ and a $\Gamma$-map $\phi: B \to A$, there is a $\Gamma$-map $\rho: C \to A$ with $\rho\circ\psi=\phi$.
A linear map $\psi : A \to B$ between $C^*$-algebras is  {\it faithful} if $\psi(a^*a) = 0$ implies $a = 0$, for $a\in A$.

Recall from the previous section that a $\Gamma$-$C^*$-algebra $A$ is {\it proximal} if the corresponding action on the state space $\mathcal S(A)$ is proximal, i.e., given states $\omega_1,\omega_2\in\mathcal S(A)$, there is a net $(t_i)$ in $\Gamma$ with $\lim_i t_i\omega_1= \lim_i t_i\omega_2$, and  it is {\it minimal} if for each pure state $\omega\in \mathcal P(A)$, the orbit $\{s\omega: s\in\Gamma\}$ is dense in $\mathcal P(A)$. 

When $\Gamma$ acts on a convex subset $K$ of a complex vector space by affine transformations, we say that $K$ is a {\it affine $\Gamma$-space}. When $K$ contains no proper affine $\Gamma$-subspace, it is called minimal. This notion of minimality is equivalent to the above notion defined by point masses \cite{gl2}. In particular, a $\Gamma$-$C^*$-algebra $A$ is minimal iff its pure state space $\mathcal P(A)$ has  no proper affine $\Gamma$-subspace. 
An affine $\Gamma$-space $K$ is said to be 
{\it irreducible} if it contains no proper closed convex $\Gamma$-invariant set.

\begin{definition} \label{bd1}
A $\Gamma$-$C^*$-algebra $A$ is called a (non commutative) $\Gamma$-boundary if its minimal and proximal. 	
\end{definition} 

Note that for a compact $\Gamma$-space $X$, proximality of $C(X)$ in the above sense means that $\Gamma$ acts proximal on the space $P(X)$ of probability measures on $X$, which is called strong proximality (for $X$) in the classical dynamics.

\begin{definition} \label{bd2}
	Let $\pi\in{\rm Rep}(G)$. A $\Gamma$-$C^*$-algebra $A$ is called a  $\pi$-boundary if it there is a $\Gamma$-embedding $\phi: A\to \mathcal{B}_{\pi}$. 	
\end{definition}

If $A$ is  a  $\pi$-boundary,  each  $\Gamma$-map $\phi: A\to \mathcal{B}_{\pi}$,  induces a surjection $\phi_{*}:  \mathcal S(\mathcal{B}_{\pi})\to \mathcal S(A)$ (since $\phi$ is u.c.p.). We say that $A$ is  a  {\it regular} $\pi$-boundary  if $\phi_{*}$ maps pure states to pure states, for each each  $\Gamma$-map $\phi$. A particular case where this holds for a map $\phi$ is when $A$ is a hereditary $C^*$-subalgebra of $\mathcal{B}_{\pi}$ (under Choi-Effros product) and $\phi: A\hookrightarrow \mathcal{B}_{\pi}$ is the inclusion map (in this case, since each state on $A$ has a {\it unique} extension to  a state on $\mathcal{B}_{\pi}$ \cite[Theorem 3.3.9]{di}, the restriction map $\phi_{*}$ sends pure states to pure states). Also by $\pi$-rigidity,  $\mathcal{B}_{\pi}$ itself is a regular $\pi$-boundary. A $\pi$-boundary $A$ is called {\it embeddable} if there is a $\Gamma$-equivariant
injective $*$-homomorphism: $A\to  \mathcal{B}_\pi$ (compare to Remark \ref{homo}). In this case, $A$ could be regarded as a $C^*$-subalgebra of $\mathcal{B}_{\pi}$ (this does not imply that $A$ is regular, since this copy is not necessarily a hereditary $C^*$-subalgebra).

By Proposition \ref{mp}, the Furstenberg-Hamana boundary $\mathcal{B}_{\pi}$ is also $\Gamma$-boundary. The next lemma relates two notions of boundary.

\begin{lemma} \label{bd} For the left regular representation $\lambda$:
	
	$(i)$ The class of regular $\lambda$-boundaries is the same as the class of $\Gamma$-boundaries.
	
	$(ii)$ Commutative $\Gamma$-boundaries are exactly $C^*$-algebras of the form $C(X)$, with $X$ a $\Gamma$-boundary in the classical sense. These boundaries are automatically regular and embeddable.
\end{lemma}
\begin{proof}
	$(i)$ We know that $\mathcal{B}_{\lambda}=C(\partial_F\Gamma)$, where $\partial_F\Gamma$ is the Furstenberg boundary of $\Gamma$ \cite[Example 3.7]{bek}. If $A$ is  a regular  $\lambda$-boundary with $\Gamma$-embedding $\phi: A\to C(\partial_F\Gamma)$, then the induced  $\Gamma$-surjection $\phi_{*}:  P(\partial_F\Gamma)\to \mathcal S(A)$ sends $\partial_F\Gamma$ to $\mathcal P(A)$, that is,  $\mathcal P(A)$ is a factor of $\partial_F\Gamma$, and so is minimal and proximal. Thus $A$ is a $\Gamma$-boundary in the sense of Definition \ref{bd1}. 
	
	Conversely if $A$ is a $\Gamma$-boundary, then by universality of the Furstenberg boundary (cf., \cite[Proposition 2.4]{gl2}), there is a $\Gamma$-surjection $\theta:  X:=\partial_F\Gamma\to Y:=\mathcal P(A)$. The induced map $\theta^*: C(Y)\to C(X); f\mapsto f\circ \theta$ is a $\Gamma$-map. On the hand, the Gelfand map $A\to C(Y); a\mapsto \hat a$ is linear, positive, and unital, and so u.c.p.,  since positive maps into commutative $C^*$-algebras are automatically c.p. (see for instance, \cite[Theorem 3.9]{p}). This  map is also  $\Gamma$-equivariant, 
	\begin{align*}\langle (s\cdot a)\hat{}, \omega\rangle &=\langle \omega,s\cdot a\rangle=\langle s^{-1}\omega,a\rangle
		=\langle \hat a, s^{-1}\omega\rangle
		=\langle s\cdot \hat a, \omega\rangle,
	\end{align*}
	and so a $\Gamma$-map. Composing these two maps, we get a $\Gamma$-map: $A\to C(X)=\mathcal{B}_{\lambda}$, which means that $A$ is a $\lambda$-boundary. Finally the induced map: $\mathcal S(\mathcal{B}_{\pi})\to \mathcal S(A)$ clearly sends $X=\partial_F\Gamma$ to $Y=\mathcal P(A)$, i.e., $A$ is also regular.
		
	$(ii)$ The first assertion follows from $(i)$ and the fact that  a compact $\Gamma$-space $X$ is strongly proximal iff the weak$^{*}$-compact $\Gamma$-space $P(X)$ of probability measures on $X$ is proximal. Since $\Gamma$-embeddings in this case are: $C(X)\hookrightarrow C(\partial_F\Gamma); f\mapsto f\circ \theta,$ for a $\Gamma$-surjection $\theta:  \partial_F\Gamma\to X$,  embeddability is automatic. Finally, reqularity follows from \cite[Proposition 4.2]{f2}.\qed    	
\end{proof}	

\begin{lemma} \label{unique} 
	If $A$ is a regular $\pi$-boundary, then there is a unique $\Gamma$-map: $A\to \mathcal B_\pi$. 
\end{lemma}
\begin{proof}
	If $A$ is a regular $\Gamma$-boundary, given two $\Gamma$-maps $b_i: A\to \mathcal B_\pi$, $i=1,2$, by regularity, we get $\Gamma$-maps $b_{i*}: \mathcal P(\mathcal B_\pi)\to \mathcal P(A)$. By convexity of pure state space, $\gamma:=\frac{1}{2}(b_{1*}+b_{2*})$ is also a $\Gamma$-map into $\mathcal P(A)$. By the fact that pure states are extreme points in the state space for unital C*-algebras \cite[Proposition 2.5.5]{di}, we get $b_{1*}=b_{2*}$, that is, $\omega\circ b_{1}=\omega\circ b_{2}$, for each $\omega\in P(\mathcal B_\pi)$, thus $b_1=b_2$, as pure states separate the points in the  unital C*-algebra $\mathcal B_\pi$.   .\qed    	
\end{proof}	

The next lemma extends \cite[Proposition 3.1]{ke} (compare also to Lemma \ref{1to1}).

\begin{lemma} \label{1-1}
There is a one-one correspondence between $\Gamma$-maps: $C^*_\pi(\Gamma)\to \mathcal{B}_{\pi}$ and $\pi$-boundaries  $A$ which are isomorphic (as a $\Gamma$-$C^*$-algebra) to a quotient of $C^*_\pi(\Gamma)$. In this correspondence, regular $\pi$-boundaries correspond to regular $\Gamma$-maps. 
\end{lemma}
\begin{proof}
For each $\Gamma$-map $\phi: C^*_\pi(\Gamma)\to \mathcal{B}_{\pi}$, we claim that $A:=\phi(C^*_\pi(\Gamma))$ is a  $\pi$-boundary. We endow $A$ with the Choi-Effros product coming from the surjective u.c.p. map $\phi: C^*_\pi(\Gamma)\to A$. Then $A$ is a unital $\Gamma$-$C^*$-algebra and $\phi$ is a surjective $*$-homomorphism. Since $\phi$ is a $\Gamma$-map, the inclusion $\iota: A\hookrightarrow \mathcal{B}_{\pi}$ is a $\Gamma$-embedding. Also, $\iota$ is regular iff $\phi$ is so. 

Conversely, if $A$ is a  $\pi$-boundary with $\Gamma$-embedding $\phi: A\to \mathcal{B}_{\pi}$, and $q: C^*_\pi(\Gamma)\to A$ is a quotient map (and a $\Gamma$-map), then   $\phi\circ q: C^*_\pi(\Gamma)\to \mathcal{B}_{\pi}$ is a $\Gamma$-map. If $\phi$ is regular, so is $\phi\circ q$, since the induced map $q_{*}: \mathcal S(A)\to \mathcal S(C^*_\pi(\Gamma))$ sends pure states to pure states. \qed	
\end{proof}	

If $\pi\in{\rm Rep}(\Gamma)$, then for each (unital) $C^*$-subalgebra $A$ of $B(H_\pi)$, $\Gamma$ acts on $A$ by Ad$_\pi$.  From now on, whenever $A\subseteq B(H_\pi)$, we regard $A$ as a $\Gamma$-$C^*$-algebra under Ad$_\pi$-action.

\begin{remark}  \label{homo}
	
	$(i)$ In \cite[Proposition 3.1]{ke}, Kennedy showed that there is a  one-one correspondence between classical $\Gamma$-boundaries in the state space
$\mathcal S(C^*_\lambda(\Gamma))$ taking a classical $\Gamma$-boundary $X\subseteq \mathcal S(C^*_\lambda(\Gamma))$
to the boundary map $b_\lambda : C^*_\lambda(\Gamma) \to C(\partial_F\Gamma)$, given by $b_\lambda(a)(x) :=\langle b_X(x), a\rangle$, for $a\in C^*_\lambda(\Gamma)$ 
and $x\in \partial_F\Gamma$, where $b_X: \partial_F\Gamma\to X$ is the canonical $\Gamma$-surjection.
In our (more general) setting, the classical $\Gamma$-boundary $X\subseteq \mathcal S(C^*_\lambda(\Gamma))$ is replaced by a regular $\pi$-boundaries  $A$ which is a quotient of $C^*_\pi(\Gamma)$, which is sent under correspondence to the boundary map $b_\pi : C^*_\pi(\Gamma) \to \mathcal{B}_{\pi}$, given by $b_\pi=\phi\circ q$, where  $\phi: A\to \mathcal{B}_{\pi}$ is a $\Gamma$-embedding and $q: C^*_\pi(\Gamma)\to A$ is a quotient map.

$(ii)$ By \cite[Remark 6.3]{kk}, if there is a $\Gamma$-map: $A\to \mathcal{B}_{\pi}$, then there is a $\Gamma$-equivariant
injective $*$-homomorphism: $A\to  \mathcal{B}_{\pi}^{**}$.

$(iii)$ Since $\Gamma$ acts on $C^*_\pi(\Gamma)$ by  Ad$_\pi$, any ideal $I\unlhd C^*_\pi(\Gamma)$ is automatically $\Gamma$-invariant, thus the quotient is a $\Gamma$-$C^*$-algebra. However, there is not necessarily a $\Gamma$-map from   $C^*_\pi(\Gamma)/I$ to $B(H_\pi)$, thus only {\it some} quotients of $C^*_\pi(\Gamma)$ could be isomorphic to $\pi$-boundaries, as in Lemma \ref{1-1}. This pathology does not happen for $\lambda$, since $C(\partial_F G)=I_\Gamma(\mathbb C)$, and so any $\Gamma$-operator system has a $\Gamma$-map into $C(\partial_F G)$ by $\Gamma$-injectivity. For $\pi$ however, $\mathcal{B}_{\pi}$ is relative $\Gamma$-injective envelop of $\mathbb C$ inside $B(H_\pi)$, and so only $\Gamma$-operator systems inside $B(H_\pi)$ have a $\Gamma$-map into $\mathcal{B}_{\pi}$ (by $\pi$-injectivity).       
\end{remark}

Given a unital $\Gamma$-$C^*$-algebra $A$, the state space $\mathcal S(A)$ is compact in the weak$^{*}$-topology. For $\omega\in \mathcal S(A)$, the stabilizer and open stabilizer subgroups of $\omega$ are $$\Gamma_\omega := \{t\in\Gamma:
t\omega = \omega\}, \ \ \Gamma^0_\omega := \{t\in\Gamma: t \ {\rm fixes\ an\ open\ neighborhood\ of\ }\omega\},$$ with $\Gamma^0_\omega\unlhd\Gamma_\omega$.  The stabilizer and open stabilizer
maps are defined by
$${\rm Stab}:\mathcal S(A)\to {\rm  Sub}(\Gamma);\ \omega\mapsto \Gamma^0_\omega, \ \ \ {\rm Stab}^0:  \mathcal S(A)\to {\rm  Sub}(\Gamma); \ \omega\mapsto \Gamma^0_\omega,$$ where ${\rm  Sub}(\Gamma)$ is endowed with the Chabauty
topology (the restriction of the product topology on $\{0, 1\}^\Gamma$). We say that  the action $\Gamma\curvearrowright A$
has {\it Hausdorff germs} if the open stabilizer
map Stab$^0$ is continuous on the pure state space $\mathcal P(A)$. The stabilizer map is continuous at a state  $\omega$ if $\Gamma_\omega=\Gamma^0_\omega$ \cite[Lemma 2.2]{lb}.

Remark \ref{homo}$(iii)$ motivates the following terminology: an ideal $I\unlhd C^*_\pi(\Gamma)$ is called a {\it $\pi$-ideal} if there is a $\Gamma$-embedding of   $C^*_\pi(\Gamma)/I$ in $B(H_\pi)$. By $\Gamma$-injectivity of $C(\partial_F G)$, every ideal $I\unlhd C^*_\lambda(\Gamma)$ is automatically a $\lambda$-ideal.

\begin{proposition} \label{sq}
	Given a $\Gamma$-map $\phi: C^*_\pi(\Gamma)\to \mathcal{B}_{\pi}$,  the left kernel $I_\phi :=
	\{x\in C^*_\pi(\Gamma): \phi(x^*x) = 0\}$ is a closed $\pi$-ideal of $C^*_\pi(\Gamma)$. Conversely, for every proper closed $\pi$-ideal
	$J\unlhd C^*_\pi(\Gamma)$, there is a  $\Gamma$-map $\psi: C^*_\pi(\Gamma)\to \mathcal{B}_{\pi}$ such that $J\subseteq I_\psi$.
	In particular, if there is a  unique $\Gamma$-map $\phi: C^*_\pi(\Gamma)\to \mathcal{B}_{\pi}$, then $I_\phi$ contains all
	proper closed $\pi$-ideals of $C^*_\pi(\Gamma)$  and $C^*_\pi(\Gamma)/I_\phi$  is the unique non-zero  quotient of $C^*_\pi(\Gamma)$ with no $\pi$-ideal. 
\end{proposition}
\begin{proof}
	We adapt the proof of \cite[Proposition 3.1]{ks}. The left kernel $I_\phi$   is a left ideal by Schwarz inequality for u.c.p.
	maps. It is also closed by the fact that c.p. maps are (c.b. and so) bounded. The equality
	$$\phi((x\pi_{t})^*(x\pi_{t}))=\phi(\pi_{t^{-1}}x^*x\pi_{t}) = t^{-1}\cdot \phi(x^*x)$$
	shows that $I_\phi$ is also a right ideal. Finally, the quotient $C^*_\pi(\Gamma)/I_\phi$ is $\Gamma$-isometric to ${\rm Im}(\phi)$, as $\Gamma$-operator systems, and ${\rm Im}(\phi)$ $\Gamma$-embeds in $B(H_\pi)$, thus $I_\phi$ is a $\pi$-ideal.  
	
	Given a proper closed $\pi$-ideal
	$J\unlhd C^*_\pi(\Gamma)$, since $C^*_\pi(\Gamma)/J$ $\Gamma$-embeds in $B(H_\pi)$, by $\pi$-injectivity, there is a
	$\Gamma$-map: $C^*_\pi(\Gamma)/J\to \mathcal{B}_{\pi}$ \cite[Proposition 3.4(5)]{kk}. Composing this  with the canonical quotient
	map $q: C^*_\pi(\Gamma)\to C^*_\pi(\Gamma)/J$, we get a $\Gamma$-map $\psi: C^*_\pi(\Gamma)\to \mathcal{B}_{\pi}$ such that $J={\rm ker}(q)\subseteq I_\psi$.
	The  rest of assertions now follow. \qed
\end{proof}	

The restriction of working only with $\pi$-ideals is annoying. This suggests to relax the situation by considering $\Gamma$-maps $\phi: C^*_\pi(\Gamma)\to \mathcal{B}_{\sigma}$ with an appropriate representation $\sigma\in{\rm Rep}(\Gamma)$. Indeed, we may construct  representation $\sigma$ in such a way that all ideals of $C^*_\pi(\Gamma)$ are $\sigma$-ideals: given a proper ideal $J\unlhd C^*_\pi(\Gamma)$, compose the canonical surjection: $C^*_{\rm max}(\Gamma)\twoheadrightarrow C^*_\pi(\Gamma)$ with the quotient map: $C^*_\pi(\Gamma)\to C^*_\pi(\Gamma)/J$ to get a representation of $C^*_{\rm max}(\Gamma)$. Let $\sigma_J$ be the associated representation of $\Gamma$, and put $\sigma:=\oplus \sigma_J$, summing over all proper ideals $J\unlhd C^*_\pi(\Gamma)$.  Then $C^*_\pi(\Gamma)/J$ embeds in $B(H_\sigma)$, and by construction, this is a $\Gamma$-embedding. This leads to the following modification of the previous proposition with an almost verbatim proof.  

\begin{proposition} \label{sq2}
	Let $\sigma$ be a representation of $\Gamma$ such that all (proper) quotients of $C^*_\pi(\Gamma)$ $\Gamma$-embed in $B(H_\sigma)$. Given a $\Gamma$-map $\phi: C^*_\pi(\Gamma)\to \mathcal{B}_{\sigma}$,  the left kernel $I_\phi :=
	\{x\in C^*_\pi(\Gamma): \phi(x^*x) = 0\}$ is a closed ideal of $C^*_\pi(\Gamma)$. Conversely, for every proper closed ideal
	$J\unlhd C^*_\pi(\Gamma)$, there is a  $\Gamma$-map $\psi: C^*_\pi(\Gamma)\to \mathcal{B}_{\sigma}$ such that $J\subseteq I_\psi$.
	In particular, if there is a  unique $\Gamma$-map $\phi: C^*_\pi(\Gamma)\to \mathcal{B}_{\sigma}$, then $I_\phi$ contains all
	proper closed ideals of $C^*_\pi(\Gamma)$  and $C^*_\pi(\Gamma)/I_\phi$  is the unique non-zero simple quotient of $C^*_\pi(\Gamma)$. 
\end{proposition}  

A (non commutative) {\it action class} of $\Gamma$ is a collection $\mathfrak C$ of unital $\Gamma$-$C^*$-algebras. A representation $\pi\in$Rep$(\Gamma)$ is a $\mathfrak C$-representation, writing $\pi\in$Rep$_{\mathfrak{C}}(\Gamma)$, if there is a $\Gamma$-map: $A \to B(H_\pi)$, for some $A\in\mathfrak{C}$. If 
$\pi$ is in Rep$_{\mathfrak{C}}(\Gamma)$, then so is $\sigma$, for each $\sigma\preceq\pi$:  there is a $\Gamma$-map: $A\to B(H_\pi)$, for some $A\in\mathfrak{C}$. Since $\sigma\preceq\pi$, there is a  $\Gamma$-map: $C^*_\pi(\Gamma)\to C^*_\sigma(\Gamma)$, which then extends to a $\Gamma$-map: $B(H_\pi)\to B(H_\sigma)$. The composition gives a $\Gamma$-map: $A\to B(H_\sigma)$, showing that $\sigma\in$Rep$_{\mathfrak{C}}(\Gamma)$.

A subgroup $\Lambda\leq \Gamma$  is a $\mathfrak C$-subgroup, writing $\Lambda\in$Sub$_{\mathfrak{C}}(\Gamma)$, 
if $\lambda_{\Gamma/\Lambda}\in$Sub$_{\mathfrak{C}}(\Gamma)$. 

For a state $\omega\in S(A)$ of a unital $\Gamma$-$C^*$-algebra $A$ the corresponding {\it Poisson map} is the $\Gamma$-map $P_\omega: A\to \ell^\infty(\Gamma)$ defined by
$P_\omega(a)(s) =\langle \omega, s^{-1}\cdot a\rangle$, 
for $s\in\Gamma$, $a\in A$.
If $\Lambda\leq \Gamma$ and $\omega$ is $\Lambda$-invariant, then we have a Poisson map $$P_\omega: A\to \ell^\infty(\Gamma/\Lambda); \ \ P_\omega(a)(s\Lambda) =\langle \omega, s^{-1}\cdot a\rangle \ (s\in\Gamma, a\in A),$$
which is well-defined, since if $s\Lambda=t\Lambda$, then $(t^{-1}s) \omega=\omega$, and so $\langle \omega, t^{-1}\cdot a\rangle=\langle \omega, s^{-1}\cdot a\rangle$, for each $a\in A$.

\begin{lemma} \label{state}
    For $\Lambda\leq \Gamma$,  
	$\Lambda\in{\rm Sub}_{\mathfrak{C}}(\Gamma)$ iff there is a $\Lambda$-invariant state on some $A\in \mathfrak{C}$.
\end{lemma}
\begin{proof}
	If $A\in \mathfrak{C}$ and $\psi: A\to  B(\ell^2(\Gamma/\Lambda))$ is a $\Gamma$-map, then for a $\Lambda$-invariant state $\omega$ on $B(\ell^2(\Gamma/\Lambda))$, let $\omega_0$ be the
	restriction of $\omega$ to the operator system $\psi(A)$, then $\omega_0\circ \psi\in \mathcal S(A)$ is a $\Lambda$-invariant state on $A$.
	Conversely, if there is a $\Lambda$-invariant state $\omega$ on some $A\in \mathfrak{C}$, by the observation before this lemma, the Poisson map
	$P_\omega: A\to \ell^\infty(\Gamma/\Lambda)$ is the desired $\Gamma$-map.\qed	
\end{proof}	

The subgroup class Sub$_{\mathfrak{C}}(\Gamma)$
is invariant under conjugation, since representations  $\lambda_{\Gamma/\Lambda}$ and $\lambda_{\Gamma/\Lambda^s}$ are  unitarily equivalent.

Let $\pi,\sigma\in{\rm Rep}(G)$ and let $A$ be a 
 unital $\Gamma$-$C^*$-algebra with $\Gamma$-map $\phi: A\to B(H_\pi)$ (i.e., $\pi\in$Rep$_{\{A\}}(\Gamma)$). Let $\psi: B(H_\pi)\to \mathcal{B}_{\sigma}$ be a $\Gamma$-map. In particular, the restriction of $\psi$ to $C^*_\pi(\Gamma)$ gives a boundary map $b_\pi^\sigma: C^*_\pi(\Gamma)\to \mathcal{B}_{\sigma}$. We assume further that 
 $$b_A=\psi\circ \phi: A\to \mathcal{B}_{\sigma}$$
 is a $\Gamma$-embedding. When $\mathcal{B}_{\sigma}$ is $\Gamma$-embedded in $B(H_\pi)$, alternatively we could also start with a boundary map $b_\pi^\sigma: C^*_\pi(\Gamma)\to \mathcal{B}_{\sigma}$ and extend it by $\pi$-rigidity to a $\Gamma$-map $\psi: B(H_\pi)\to \mathcal{B}_{\sigma}$ and get $b_A$ as above (we still need the assumption that $b_A$ is an embedding, i.e., a contraction). This situation is quite typical in this section, and we frequently use the following key fact that is an adaptation of \cite[Theorem 3.5]{ks}. For a unital $C^*$-algebra $B$ and element $x\in B$ we use the following notation: 
 $${\rm supp}(x):=\overline{\{\omega\in \mathcal P(B): \langle \omega, x\rangle\neq 0\}},$$
 where the closure is taken in weak$^*$-topology of $B^*$.  
    
\begin{lemma} \label{key}
    	Let $\psi$, $\phi$, $b_\pi^\sigma$ and $b_A$ are $\Gamma$-maps as above. Let $\mathcal P(A)$ be the pure state space of $A$ and for $s\in\Gamma$, let 
    	$$\mathcal P^s(A):=\{\omega\in \mathcal P(A): s\omega=\omega\}$$
    	be the fixed point set of $s$, and $\mathcal P^s_0(A)$ be the interior of $\mathcal P^s(A)$ in weak$^*$-topology of $A^*$. If the $\Gamma$-embedding $b_A=\psi\circ \phi: A\to \mathcal{B}_{\sigma}$ is regular, then 
    	$${\rm supp}(b_\pi^\sigma(\pi_s))\subseteq \overline{b_{A*}^{-1}(\mathcal P^s_0(A))},$$
    	where $b_{A*}:\mathcal P(\mathcal{B}_{\sigma})\to \mathcal P(A)$ is the induced map on pure state spaces, and  closure on the right hand side is taken in the weak$^*$-topology of $\mathcal{B}_{\sigma}^*$.   	 
\end{lemma}
\begin{proof}
    	If  $\psi: B(H_\pi)\to  B(\ell^2(\Gamma/\Lambda))$ and $\phi: A\to B(H_\pi)$ are $\Gamma$-maps, and $b_A=\psi\circ \phi$. Let $b_\pi^\sigma$ be the restriction of $\psi$ to $C^*_\pi(\Gamma)$. Arguing by the way of contradiction, let there is $s\in\Gamma$ and $\omega\in \mathcal P(A)$ with 
    	$\omega\notin \overline{b_{A*}^{-1}(\mathcal P^s_0(A))},$ such that $\langle\omega, b_\pi^\sigma(\pi_s)\rangle\neq 0.$ Choose a neighborhood $U$ of $\omega$ in $\mathcal P(A)$ with $U\cap b_{A*}^{-1}(\mathcal P^s_0(A))=\emptyset$ such that $\langle\varrho, b_\pi^\sigma(\pi_s)\rangle\neq 0,$ for each $\varrho\in U$. Choose $\omega_0\in U$ with $b_{A*}(\omega_0)\notin P^s(A)$. 
    	
    	By Proposition \ref{mp}, the translates of $U$ by elements of $\Gamma$ cover $\mathcal P(\mathcal{B}_{\sigma})$, and so by compactness, $\mathcal P(\mathcal{B}_{\sigma})\subseteq \bigcup_{t\in F} tU$, for some finite set $F\subseteq \Gamma$. By regularity, $b_{A*}(\mathcal{B}_{\sigma}) \subseteq \mathcal P(A)$, thus by the fact that $b_{A*}$ is a $\Gamma$-map, $\mathcal P(A)\subseteq \bigcup_{t\in F} tb_{A*}(U)$. In particular, $b_{A*}(U)$ has non-empty interior, by Baire category theorem. By Krein-Hahn-Banach theorem, there is $0\leq a\leq 1$ in $A$ with
    	$$ \langle b_{A*}(\omega_0), a\rangle=1, \ \ \langle b_{A*}(\omega_0), s\cdot a\rangle=0.$$  
    	Thus,
    	$$\langle\omega_0, \psi\circ\phi(a)\rangle=\langle\omega_0, b_A(a)\rangle=\langle b_{A*}(\omega_0), a\rangle=1,$$
    	whereas,
    	$$\langle\omega_0, \psi(\pi_s\phi(a)\pi_s^*)\rangle=\langle\omega_0, \psi\circ\phi(s\cdot a)\rangle=\langle b_{A*}(\omega_0), s\cdot a\rangle=0.$$
    	Since $\phi$ is u.c.p., we have $0\leq \phi(a)\leq 1$, so arguing as in the proof of Lemma \ref{rad2} (cf., \cite[Lemma 2.2]{hk}), we get $0 \leq \pi_s^*\phi(a)\pi_s\leq 1$ and $\langle\omega_0, \psi(\pi_s\phi(a)^2\pi_s^*)\rangle=0.$
    	Similarly, from $\langle\omega_0, \psi(1-\phi(a))\rangle=0,$ it follows that $\langle\omega_0, \psi((1-\phi(a))^2)\rangle=0.$ By Cauchy-Schwartz inequality,
    	$$\big|\langle\omega_0, \psi((1-\phi(a))\pi_s)\rangle\big|\leq \langle\omega_0, \psi((1-\phi(a))^2)\rangle\langle\omega_0, \psi(1)\rangle=0.$$ Thus, for the state $\tau:=\omega_0\circ \psi$ on $B(H_\pi)$,
    	\begin{align*}
    		\big|\langle\omega_0, \psi(\phi(a)\pi_s)\rangle\big|&=\big|\tau(\phi(a)\pi_s)\big|=\big|\tau( \pi_s^*\pi_s\phi(a)\pi_s)\big|\\&=\big|\tau(\pi_s^*\phi(a)\pi_s\pi_s^*)\big|\leq \big|\tau(\pi_s^*\phi(a)^2\pi_s)\big|
    		=0,
    	\end{align*} 
    	Hence, $$\langle\omega_0, b_\pi^\sigma(\pi_s)\rangle=\langle\omega_0, \psi(\pi_s)\rangle=\langle\omega_0, \psi(\phi(a)\pi_s)\rangle+\langle\omega_0, \psi((1-\phi(a))\pi_s)\rangle=0,$$
    	which contradicts the fact that $\omega_0\in U$ and $\langle\varrho, b_\pi^\sigma(\pi_s)\rangle\neq 0,$ for each $\varrho\in U$.\qed	
    \end{proof}	

\begin{corollary} \label{supp}
	Let $A$ be a $\pi$-boundary, then every trace on $C^*_\pi(\Gamma)$ is supported on $\ker(\Gamma\curvearrowright A)$.
\end{corollary}
\begin{proof}
	Since there is a $\Gamma$ embedding of $\mathbb C$ into $\mathcal{B}_{\pi}$, a given trace $\tau$ on $C^*_\pi(\Gamma)$ may be regarded as a $\Gamma$-map into $\mathcal{B}_{\pi}$ (still denoted by $\tau$) and could be extended to a $\Gamma$-map $\psi: B(H_\pi)\to \mathcal{B}_{\pi}$. Let $s\notin\ker(\Gamma\curvearrowright A)$ and choose $a\in A$ with $s\cdot a\neq a$. Choose $\omega_0\in \mathcal P(A)$ with $\langle\omega_0,s\cdot a\rangle\neq \langle\omega_0, a\rangle$, then $\omega_0\notin \mathcal P^s(A)$. Since $b_{A*}:  \mathcal P(\mathcal{B}_{\pi})\to \mathcal P(A)$ is onto and $\mathcal P^s(A)$ is weak$^*$-closed,
	$$\overline{b_{A*}^{-1}(\mathcal P^s_0(A))}\subseteq b_{A*}^{-1}(\overline{\mathcal P^s_0(A)})\subseteq b_{A*}^{-1}({\mathcal P^s(A)})\subsetneq \mathcal P(\mathcal{B}_{\pi}).$$
	By Lemma \ref{key}, applied to the case where $\sigma=\pi$, $b_\pi^\pi=\tau$, and $\phi$ is the $\Gamma$-embedding: $A\hookrightarrow \mathcal{B}_{\pi}\hookrightarrow B(H_\pi)$, by the last proper inclusion above, we get ${\rm supp}(\tau(\pi_s))\neq \mathcal P(\mathcal{B}_{\pi})$. Choose $\omega\notin{\rm supp}(\tau(\pi_s))$, then 
	$$\tau(\pi_s)=\langle\omega,\tau(\pi_s)1\rangle=0.\ \ \ \qed$$   
\end{proof}

\begin{corollary} \label{inv}
	If $C^*_\pi(\Gamma)$ is nuclear and has a trace, then each $\pi$-boundary has a $\Gamma$-invariant state.
\end{corollary}
\begin{proof}
	Take a trace $\tau$ on $C^*_\pi(\Gamma)$, which is then amenable by \cite[Theorem 4.2.1]{br}. Extending $\tau$ to a hypertrace $\psi$ on $B(H_\pi)$ and composing $\psi$ with $\Gamma$-embedding: $A\hookrightarrow \mathcal{B}_{\pi}\hookrightarrow B(H_\pi)$, we  get a $\Gamma$-invariant state on the $\pi$-boundary $A$.  	
\end{proof} 
  
  \begin{remark} \label{map}
  	$(i)$ Since there is a
  	${\rm Rad}_\pi(\Gamma)$-invariant state on $B(H_\pi)$, which restricts to a 
  	${\rm Rad}_\pi(\Gamma)$-invariant state on any $\pi$-boundary $A$. By normality of 
  	${\rm Rad}_\pi(\Gamma)$ and proximality of $\Gamma\curvearrowright A$,  ${\rm Rad}_\pi(\Gamma)$ acts trivially on $A$. In particular, ${\rm Rad}_\pi(\Gamma)$ is contained in $N_\pi:=\bigcap_{A}\ker(\Gamma\curvearrowright A)$, where the intersection runs over all $\pi$-boundaries $A$.
  	
  	$(ii)$ As the proof of Corollary \ref{inv} reveals, we do not need a $\Gamma$-embedding: $A\hookrightarrow \mathcal{B}_{\pi}$, but just  a $\Gamma$-map: $A\to B(H_\pi)$. Also we only need to assume that $C^*_\pi(\Gamma)$  has an amenable trace. When $\pi=\lambda$, all traces on $C^*_\lambda(\Gamma)$ (if any) are automatically amenable, and indeed, if there is a trace, then $C^*_\lambda(\Gamma)$ is nuclear \cite[Proposition 4.1.1]{br}.  When $\pi$ is the universal representation, the trivial representation can be regarded as a  locally finite 
  	dimensional trace on $C^*_{\rm max}(\Gamma)$. 
  	There is another trace on  $C^*_{\rm max}(\Gamma)$ coming from the left regular representation 
  	and, in contrast to the reduced  case, this trace is also amenable for a 
  	large class of non-amenable discrete groups, e.g., for residually finite groups \cite[Proposition 4.1.4]{br}.   
  	\end{remark}
  
  \begin{example} \label{ex4}
  	$(i)$ Let $(\Gamma, A,\alpha)$ be a $C^*$-dynamical system with $A$ unital. A covariant representation of $(\Gamma, A,\alpha)$ is a triple $(\theta, \pi, H)$ where $\theta$ and $\pi$
  	 are respectively a representation of $A$ and a unitary representation of $\Gamma$ in
  	the same Hilbert space $H$, satisfying the covariance rule
  	$$\pi_t\theta(a)\pi_t^* = \theta(\alpha_t(a)),$$
  	for $t\in \Gamma$ and $a\in A$. The covariance condition simply means that $\theta: A\to B(H)$ is a $\Gamma$-map w.r.t. $\alpha$-action on $A$ and Ad$_\pi$-action on $B(H)$. It now follows from Remark \ref{map}$(ii)$ that $A$ has a $\Gamma$-invariant state when $\C^*_\pi(\Gamma)$ has an amenable trace.   
  	
  	$(ii)$ A covariant representation $(\theta, \pi, H)$ gives rise to a *-homomorphism $\theta\rtimes \pi$ from $A[\Gamma]$
  	into $B(H)$, defined by
  	$$(\theta\rtimes \pi)(\sum a_tt) =\sum \theta(a_t)\pi_t.$$
  	Let $A\rtimes_{(\theta,\pi)}\Gamma$ be the $C^*$-algebra generated by the range of this integrated representation, then $A\rtimes_{(\theta,\pi)}\Gamma$ is $\Gamma$-$C^*$-subalgebra of $B(H)$ under Ad$_\pi$-action, and so it  has a $\Gamma$-invariant state when $\C^*_\pi(\Gamma)$ has an amenable trace.
  	
  	$(iii)$  When the boundary $\mathcal{B}_{\pi}$ is C*-embeddable, we let $\Gamma$ act on $\mathcal{B}_{\pi}$ by Ad$_\pi$, and so the embedding $\iota: \mathcal{B}_{\pi}\hookrightarrow B(H_\pi)$   would be a $\Gamma$-map. With the notations of Theorem \ref{emb}, let $\tilde{\mathcal{B}}_{\pi}$ denote the C*-algebra generated by $\mathcal{B}_{\pi}\cup \pi(\Gamma)$
  	in $B(H_\pi)$, then $\tilde{\mathcal{B}}_{\pi}=\mathcal{B}_{\pi}\rtimes_{(\iota,\pi)}\Gamma$. In particular, $\tilde{\mathcal{B}}_{\pi}$ has a $\Gamma$-invariant state when $\C^*_\pi(\Gamma)$ has an amenable trace.
  \end{example}

Next we discuss uniqueness of boundary maps. In Example \ref{ex4}$(i)$, following \cite[Definition 4.1]{ks}, a nondegenerate covariant representation  $(\theta, \pi, H)$ is called {\it germinal} if 
$$\pi_s\theta(a)=\theta(a)\ \ {\rm whenever} \ {\rm supp}(a)\subseteq \mathcal P^s(A),$$
  or equivalently, if this holds for any $a\in A$ with ${\rm supp}(a)\subseteq \mathcal P_0^s(A)$ (as the latter subset of $A$ is dense in the former).
  
  \begin{example} \label{ger}
  	$(i)$ The covariant representation $(\kappa_\nu, M)$ is germinal, where $\nu$ is a $\sigma$-finite quasi-invariant measure on a compact $\Gamma$-space $X$, $\kappa_\nu$ is the Koopman representation of $\Gamma$ on $L^2(X, \nu)$ and $M: C(X) \to B(L^2(X, \nu))$
  	is the multiplication representation \cite[Proposition 4.3]{ks}.

  	$(ii)$ If $\Lambda\leq \Gamma$ and $\omega$ is a $\Lambda$-invariant state on a $\Gamma$-C*-algebra $A$, then The pair $(\lambda_{\Gamma/\Lambda}, P_\omega)$ is a covariant pair whenever $\Gamma^0_\omega\leq \Lambda\leq \Gamma_\omega$, where 
  	$$P_\omega: A\to \ell^\infty(\Gamma/\Lambda); \ \ P_\omega(a)(t\Lambda) =\langle \omega, t^{-1}\cdot a\rangle \ (t\in\Gamma, a\in A),$$
  	 is the Poisson map: given $s,t\in \Gamma$ and $a\in A$  with ${\rm supp}(a)\subseteq \mathcal P_0^s(A),$ we have 
  	 $$P_\omega(a)(t\Lambda) =\langle \omega, t^{-1}\cdot a\rangle=\langle t\omega,  a\rangle=0,$$ if $t\omega\notin\mathcal P_0^s(A).$ In this case, we also have $st\omega\notin s\mathcal P_0^s(A)=\mathcal P_0^s(A),$ where the last equality follows from the fact that $\omega\mapsto s\omega$ is a homeomorphism. Thus,   $$\lambda_{\Gamma/\Lambda}(s)P_\omega(a)(t\Lambda)=P_\omega(a)(st\Lambda)=\langle st\omega,  a\rangle=0=P_\omega(a)(t\Lambda).$$ 
  	 On the other hand, if $t\omega\in\mathcal P_0^s(A)$, $t^{-1}st$ fixes the weak$^*$-open neighborhood $t^{-1}\mathcal P_0^s(A)$ of $\omega$, that is, $t^{-1}st\in \Gamma_\omega^0\subseteq\Lambda$, and again  we have, $$\lambda_{\Gamma/\Lambda}(s)P_\omega(a)(t\Lambda)=P_\omega(a)(st\Lambda)=P_\omega(a)(t\Lambda).$$
  \end{example}

The next result is one of the main results of this section, which extends \cite[Theorem 4.4]{ks} for the case $\pi=\lambda$ and $A=C(X)$ (the extension being non trivial, as the proof in \cite{ks} uses extensively the fact that $\mathcal B_\lambda=C(\partial_F\Gamma)$ is commutative). 

\begin{theorem}
	 Let $A$ be a regular embeddable $\pi$-boundary and $(\theta, \pi, H)$ be a germinal covariant representation, then there is a unique  boundary map $\psi:  A\rtimes_{\theta \rtimes\pi}\Gamma\to \mathcal B_\pi$, whose restriction to $C^*_\pi(\Gamma)$ is a
	unique  boundary map on $C^*_\pi(\Gamma)$.
	\end{theorem}
\begin{proof}
 Since $A$ is an embeddable  $\pi$-boundary, we may chosse a $\Gamma$-embedding $b_A: A\hookrightarrow\mathcal B_\pi$ which is also a $*$-homomorphism w.r.t. the Choi-Effros product on $\mathcal B_\pi$. By $\pi$-injectivity, this extends to a $\Gamma$-map $\psi: B(H_\pi)\to \mathcal B_\pi$. By regularity of $A$ and Lemma \ref{unique}, $\psi\circ\theta=b_A$. In particular, $\theta(A)$ is contained in the multiplicative domain of
 $\psi$, therefore by \cite[Proposition 1.5.7(2)]{bo},
  $$\psi\big(\sum_{s} \theta(a_s)\pi_s\big)=\sum_{s}\psi(\theta(a_s))\psi(\pi_s)=\sum_{s}b_A(a_s)\psi(\pi_s),$$
  for each finite sum over $\Gamma$, which shows that the restriction of $\psi$ on $A\rtimes_{\theta \rtimes\pi}\Gamma$ is unique if its restriction to $C^*_\pi(\Gamma)$ is so. Conversely, Since any boundary map on $C^*_\pi(\Gamma)$ extends to a boundary map on $A\rtimes_{\theta \rtimes\pi}\Gamma$. It remains to show that the restriction of $\psi$ to $C^*_\pi(\Gamma)$ is indeed unique. 
  
  The $\Gamma$-map $b_A: A\to \mathcal B_\pi$ induces a $\Gamma$-map $b_{A*}: \mathcal S(\mathcal B_\pi)\to \mathcal S(A)$, which sends pure states to pure states by regularity. Let $s\in \Gamma$ and $\omega\in b_{A*}^{-1}( \mathcal P_0^s(A))=:\Delta_s$. By Hahn-Banach theorem, we may choose $a\in A$ with $\|a\|=1$ such that $$\langle b_{A*}(\omega), a\rangle=1, \ \ {\rm supp}(a)\subseteq \mathcal P^s(A).$$
  Put $b:=b_A(a)$, then $\|b\|=1$, as $b_A$ is an injective $*$-homomorphism (and so an isometry). Since $(\theta, \pi, H)$ is  germinal by assumption, $\pi_s\theta(a)=\theta(a)$. Therefore, 
  $$\psi(\pi_s)b=\psi(\pi_s)b_A(a)=\psi(\pi_s)\psi(\theta(a))=\psi(\pi_s\theta(a))=\psi(\theta(a))=b_A(a)=b,$$
 where the third equality follows from the fact that the range of $\theta$ is in the multiplicative domain of $\psi$. Evaluating $\omega$ on both sides, we get $$\langle\omega,\psi(\pi_s)b\rangle=\langle\omega, b\rangle=\langle\omega, b_A(a)\rangle	=\langle b_{A*}(\omega), a\rangle=1.$$
 Since $\psi$ is u.c.p. and so  contractive, $\|\psi(\pi_s)\|\leq 1$. By Cauchy-Schwartz inequality,
 $$1=\langle\omega,\psi(\pi_s)b\rangle\leq\langle\omega,\psi(\pi_s)^*\psi(\pi_s)\rangle^{\frac{1}{2}}\langle \omega, b^*b\rangle^{\frac{1}{2}}\leq\|\psi(\pi_s)\|^2\|b\|^2\leq 1.$$
 Thus $\langle\omega,\psi(\pi_s)^*\psi(\pi_s)\rangle\langle \omega, b^*b\rangle=1$. The product of two positive real numbers, each at most one, is equal to one, and so both are one, that is, $\langle\omega,\psi(\pi_s)^*\psi(\pi_s)\rangle=\langle \omega, b^*b\rangle=1$. Since  $\langle\omega,\psi(\pi_s)b\rangle=1$ is a real number, we have $\langle\omega,b^*\psi(\pi_s)^*\rangle=1$. Repeating the above argument, we also get  $\langle\omega, bb^*\rangle=1$. Since $ \langle\omega, b\rangle=1$, these together mean that $b$ is in the multiplicative domain of $\omega$. Therefore, $1= \langle\omega,\psi(\pi_s)b\rangle=\langle\omega,\psi(\pi_s)\rangle\langle\omega,b\rangle=\langle\omega,\psi(\pi_s)\rangle$. 
 Summing up, $\langle\omega,\psi(\pi_s)\rangle=1$, for  $\omega\in\Delta_s$, and so $\langle\omega,\psi(\pi_s)\rangle=1$, for   $\omega\in\bar{\Delta_s}$, where $\bar{\Delta_s}$ is the closure of $\Delta_s$ in weak$^*$-topology. By Lemma \ref{key} (for
 $\sigma=\pi$, $\phi=\theta$, $b_\pi^\pi=\psi|_{C^*_\pi(\Gamma)}$), we have ${\rm supp}(\psi(\pi_s))\subseteq \bar{\Delta_s}$, that is, $\langle\omega,\psi(\pi_s)\rangle=0$, for   $\omega\notin\bar{\Delta_s}$. Since pure states of the unital C*-algebra $\mathcal B_\pi$ separate the points, this completes the proof. \qed     
\end{proof} 

We need the next lemma proved by Archbold and Spielberg \cite[Lemma 2]{as}. Recall that two representations of a C*-algebra $A$ are called {\it disjoint} if they have  non trivial equivalent subrepresentations \cite[5.2.2]{di}. 

\begin{lemma} \label{vanish}
	Let $A$ be a $\Gamma$-C*-algebra with action $\alpha$, $\rho: A\to B(H)$ be a
	non degenerate representation of $A$, and $\phi: A\rtimes\Gamma\to B(H)$ be a c.c.p. map extending $\theta$ on the maximal crossed product. If $\rho$ and $\rho\circ \alpha_t$ are disjoint representations, then  $\phi(a\delta_t)=0$, for
	each $a\in A$, where $a\delta_t$ is a singly supported function in $c_c(\Gamma, A)\subseteq A\rtimes\Gamma$.
\end{lemma}	

\begin{proposition} \label{vanish2}
	Let $(\theta, \pi, H)$ be a  covariant representation of C*-dynamical system $(\Gamma, A,\alpha)$ and let $\theta\rtimes\pi: A\rtimes\Gamma\to B(H)$ be the corresponding integrated representation of the maximal (full) crossed product. Let $A\rtimes_{\theta \rtimes\pi}\Gamma$ be the C*-subalgebra of $B(H)$ generated by the range of $\theta\rtimes\pi$, and $q_\theta: A\rtimes\Gamma \to A\rtimes_{\theta \rtimes\pi}\Gamma$ be the corresponding quotient map. 
	
	$(i)$ If $\theta$ and $\theta\circ \alpha_t$ are disjoint representations, then  $(\theta \rtimes\pi)(a\delta_t)=0$, for each $a\in A$.
	
	$(ii)$  If $\theta$ is faithful and the action $\Gamma\curvearrowright A$ is topologically free, there is a conditional expectation $\mathbb E_\theta: A\rtimes_{\theta \rtimes\pi}\Gamma\to \theta(A)$ such that the following diagram  commutes,

	\begin{center}
		\begin{tikzcd}
			A\rtimes\Gamma \arrow{r}{\theta \rtimes\pi} \arrow{d}[swap]{\mathbb E} &A\rtimes_{\theta \rtimes\pi}\Gamma \arrow{d}{\mathbb E_\theta} \\   
			A \arrow[swap]{r}{\theta} & \theta(A)
		\end{tikzcd}
	\end{center}  
	where $\mathbb E: A\rtimes\Gamma\to A$ is the canonical conditional expectation on the maximal (full) crossed product.
	
	$(iii)$ When  $\theta$ is faithful and the action $\Gamma\curvearrowright A$ is topologically free, for each closed ideal $I\unlhd  A\rtimes_{\theta \rtimes\pi}\Gamma$, if $I\cap \theta(A)=0$, then $I\subseteq \ker \mathbb E_{\theta}$.  
	
\end{proposition} 
\begin{proof}
	$(i)$. This follows from Lemma \ref{vanish} for $\phi=\theta \rtimes\pi$. 
	
	$(ii)$ Let  $\tilde H := \ell^2(G, H)$ and $(\tilde\sigma, \tilde\lambda, \tilde H)$ be  
	be the  representation induced by $\theta$, defined by $$\tilde\theta(a)\xi(t) = \theta(\alpha_{t^{-1}}(a))\xi(t),\ \ \tilde\lambda_s\xi(t) =\xi(s^{-1}t),$$
	for $a\in A$ and $\xi\in \tilde H$.
	Since $\theta$ is faithful,  $\Vert x\Vert_{\rm red} =\Vert(\tilde\theta \rtimes\tilde\lambda)(x)\Vert$, for
	each $x\in c_c(\Gamma, A)$ \cite[Definition 7.7]{w}. For $s\in \Gamma$, consider the partial isometry $V: \tilde H\to H; \ \xi\mapsto \xi(e)$, then for $\eta\in H$, $\tilde\lambda_sV^*\eta=\tilde\lambda_s(\eta\delta_e)=\eta\delta_e$, thus for $x = \sum_s a_s\delta_s$, 
	$$ V\big((\tilde\theta \rtimes\tilde\lambda)(x)\big)V^*\eta=V\sum_s \tilde\theta(a_s)\tilde\lambda_sV^*\eta = \sum_s \delta_s(e)\theta(a_s)\eta = \theta(a_e)\eta=\theta(\mathbb E(x))\eta,$$ where 
	$$\mathbb E: A\rtimes\Gamma\to A; \ \ \sum_s a_s\delta_s\mapsto a_e,$$
	is the canonical conditional expectation on the full crossed product. Since the action $\Gamma\curvearrowright A$ is topologically free, for each ideal $I\unlhd  A\rtimes\Gamma$, if $A\cap I=0$, then $I\subseteq \ker(\tilde\theta \rtimes\tilde\lambda)$ \cite[Theorem 1]{as}. Since $\theta$ is faithful, $A\cap\ker(\theta \rtimes\pi)=0$, thus $\ker(\theta \rtimes\pi)\subseteq \ker(\tilde\theta \rtimes\tilde\lambda)$. In particular, there is a well-defined c.p. map  $\mathbb E_\theta: A\rtimes_{\theta \rtimes\pi}\Gamma\to \theta(A)$ satisfying $\mathbb E_\theta\circ(\theta \rtimes\pi)=\theta\circ \mathbb E.$
	Now, for $x\in c_c(\Gamma, A)$,
	$$\|\mathbb E_\theta\big((\theta \rtimes\pi)(x)\big)\|=\|\theta\big(\mathbb E(x)\big)\|=\|V\big((\tilde\theta \rtimes\tilde\lambda)(x)\big)V^*\|\leq \|(\tilde\theta \rtimes\tilde\lambda)(x)\|\leq \|(\theta \rtimes\pi)(x)\|,$$
	where the last equality follows from the fact that $\ker(\theta \rtimes\pi)\subseteq \ker(\tilde\theta \rtimes\tilde\lambda)$. This means that $\mathbb E_\theta$ is c.c.p. Moreover,  it follows from $\mathbb E_\theta(\sum_s \theta(a_s)\pi_s)=\theta(a_e)$ that $\mathbb E_\theta^2=\mathbb E_\theta$, therefore, $\mathbb E_\theta$ is a conditional expectation by Tomiyama's theorem (cf., \cite[Theorem 1.5.10]{bo}). 
	
	$(iii)$ We adapt the proof of \cite[Theorem 1]{as}. Let $I$ be a closed ideal with $I\nsubseteq \ker \mathbb E_{\theta}$. Choose $a\in I$ with $\mathbb E_{\theta}(a)\neq 0$. Choose $b\in \sum_{s\in F} b_s\delta_s\in c_c(\Gamma, A)$ with $\|(\theta\rtimes \pi)(b)-a\|<\frac{1}{2}\|\mathbb E_{\theta}(a)\|.$ Let $X:=\bigcap_{s\in F\backslash\{e\}} \{\sigma\in \hat A: t\sigma\neq \sigma\}$, and for $\sigma\in X$, let $\tilde\sigma$ be $*$-homomorphism obtained as the composition of maps
	$$\theta(A)+I\to \frac{\theta(A)+I}{I}\simeq \frac{\theta(A)}{\theta(A)\cap I}=\theta(A)\xrightarrow[]{\theta^{-1}} A \xrightarrow[]{\sigma} B(H_\sigma),$$
	where $\theta(A)+I$ is closed (and so a C*-subalgebra) in $A\rtimes_{\theta \rtimes\pi}\Gamma$, as the sum of a C*-subalgebra and a closed ideal \cite[1.8.4]{di}. Extend $\tilde\sigma$ to a c.p. map $\phi: A\rtimes_{\theta \rtimes\pi}\Gamma\to B(H_\sigma)$, and observe that since $\sigma\circ\alpha_t\neq \sigma$, for $t\neq e$, Lemma \ref{vanish} applies to $\phi$ and for $x:=(\theta\rtimes\pi)(b)$, by part $(ii)$ we get,
	\begin{align*}
		\phi(x)=\phi\big((\theta\rtimes\pi)(b)\big)&=\sum_s\phi(\theta(b_s)\pi_s)=\phi(\theta(b_e))\\&=\phi(\theta\circ\mathbb E(b))=\phi\big(\mathbb E_\theta\circ(\theta\rtimes\pi)(b)\big)=\phi(\mathbb E_\theta(x)),
	\end{align*}
	and since $\phi=\sigma$ on $\theta(A)$ and 	
	$\phi(a)=\tilde\sigma(a)=0$,
	$$\|\sigma(\mathbb E_\theta(x))\|=\|\phi(\mathbb E_\theta(x))\|=\|\phi(x)\|=\|\phi(x-a)\|\leq\|x-a\|.$$ Since this holds for each $\sigma\in X$ and $X\subseteq \hat A$ is dense, we get $\|\mathbb E_\theta(x)\|\leq\|x-a\|,$ thus
	$$\|\mathbb E_\theta(a)\|\leq \|\mathbb E_\theta(x-a)\|+\|\mathbb E_\theta(x)\|\leq\|x-a\|+\|x-a\|,$$
	which contradicts the choice of $b$.  \qed	
\end{proof}	

\begin{remark} \label{rem}
	Let  $\pi\in$Rep$(\Gamma)$, then it is desirable to show that $\Gamma$ is $C^*_\pi$-simple iff it has a topologically free $\pi$-boundary. However, at this stage we don't know if this is true. With the notation of the above result, if $\theta$ is fauthful and the action $\Gamma\curvearrowright A$ is topologically free, and moreover the restriction $\psi$ of $\mathbb E_\theta$ to $C^*_\pi(\Gamma)$, considered as a linear scalar valued map is faithful (that is, if $C^*_\pi(\Gamma)_{+}\cap \ker \mathbb E_{\theta}=0$) then it follows from \cite[Proposition 3.1]{ks} that $\Gamma$ is $C^*_\pi$-simple. 
\end{remark}

%----------------------------------------------------------------------------------------

%----------------------------------------------------------------------------------------

\end{document}